\documentclass[12pt,leqno]{article}
\usepackage{amsfonts}
\usepackage{amsmath, dsfont, amsthm, amsfonts, amssymb, color}
\usepackage{mathrsfs}
\usepackage{color}
\usepackage{stmaryrd}
\newtheorem{thm}{Theorem}[section]

\newtheorem{lem}[thm]{Lemma}

\newtheorem{rem}[thm]{Remark}
\theoremstyle{definition}
\newtheorem{defn}{Definition}[section]
\newcommand{\scr}[1]{\mathscr #1}
\definecolor{wco}{rgb}{0.5,0.2,0.3}

\numberwithin{equation}{section} \theoremstyle{remark}

\newcommand{\ua}{\uparrow}

\title{{\bf
  Coupling by Change of Measure for Conditional McKean-Vlasov SDEs and Applications}\footnote{Supported in
 part by  National Key R\&D Program of China (No. 2022YFA1006000) and NNSFC (12271398).} }
\author{
{\bf   Xing Huang  }\\
\footnotesize{ Center for Applied Mathematics, Tianjin
University, Tianjin 300072, China}\\
\footnotesize{  xinghuang@tju.edu.cn}}
\begin{document}
\allowdisplaybreaks
\def\R{\mathbb R}  \def\ff{\frac} \def\ss{\sqrt} \def\B{\mathbf
B} \def\W{\mathbb W}
\def\N{\mathbb N} \def\kk{\kappa} \def\m{{\bf m}}
\def\ee{\varepsilon}\def\ddd{D^*}
\def\dd{\delta} \def\DD{\Delta} \def\vv{\varepsilon} \def\rr{\rho}
\def\<{\langle} \def\>{\rangle} \def\GG{\Gamma} \def\gg{\gamma}
  \def\nn{\nabla} \def\pp{\partial} \def\E{\mathbb E}
\def\d{\text{\rm{d}}} \def\bb{\beta} \def\aa{\alpha} \def\D{\scr D}
  \def\si{\sigma} \def\ess{\text{\rm{ess}}}
\def\beg{\begin} \def\beq{\begin{equation}}  \def\F{\scr F}
\def\Ric{\text{\rm{Ric}}} \def\Hess{\text{\rm{Hess}}}
\def\e{\text{\rm{e}}} \def\ua{\underline a} \def\OO{\Omega}  \def\oo{\omega}
 \def\tt{\tilde} \def\Ric{\text{\rm{Ric}}}
\def\cut{\text{\rm{cut}}} \def\P{\mathbb P} \def\ifn{I_n(f^{\bigotimes n})}
\def\C{\scr C}      \def\aaa{\mathbf{r}}     \def\r{r}
\def\gap{\text{\rm{gap}}} \def\prr{\pi_{{\bf m},\varrho}}  \def\r{\mathbf r}
\def\Z{\mathbb Z} \def\vrr{\varrho}
\def\L{\scr L}\def\Tt{\tt} \def\TT{\tt}\def\II{\mathbb I}
\def\i{{\rm in}}\def\Sect{{\rm Sect}}  \def\H{\mathbb H}
\def\M{\scr M}\def\Q{\mathbb Q} \def\texto{\text{o}}
\def\Rank{{\rm Rank}} \def\B{\scr B} \def\i{{\rm i}} \def\HR{\hat{\R}^d}
\def\to{\rightarrow}\def\l{\ell}\def\iint{\int}
\def\EE{\scr E}\def\Cut{{\rm Cut}}
\def\A{\scr A} \def\Lip{{\rm Lip}}
\def\BB{\scr B}\def\Ent{{\rm Ent}}\def\L{\scr L}
\def\R{\mathbb R}  \def\ff{\frac} \def\ss{\sqrt} \def\B{\mathbf
B}
\def\N{\mathbb N} \def\kk{\kappa} \def\m{{\bf m}}
\def\dd{\delta} \def\DD{\Delta} \def\vv{\varepsilon} \def\rr{\rho}
\def\<{\langle} \def\>{\rangle} \def\GG{\Gamma} \def\gg{\gamma}
  \def\nn{\nabla} \def\pp{\partial} \def\E{\mathbb E}
\def\d{\text{\rm{d}}} \def\bb{\beta} \def\aa{\alpha} \def\D{\scr D}
  \def\si{\sigma} \def\ess{\text{\rm{ess}}}
\def\beg{\begin} \def\beq{\begin{equation}}  \def\F{\scr F}
\def\Ric{\text{\rm{Ric}}} \def\Hess{\text{\rm{Hess}}}
\def\e{\text{\rm{e}}} \def\ua{\underline a} \def\OO{\Omega}  \def\oo{\omega}
 \def\tt{\tilde} \def\Ric{\text{\rm{Ric}}}
\def\cut{\text{\rm{cut}}} \def\P{\mathbb P} \def\ifn{I_n(f^{\bigotimes n})}
\def\C{\scr C}      \def\aaa{\mathbf{r}}     \def\r{r}
\def\gap{\text{\rm{gap}}} \def\prr{\pi_{{\bf m},\varrho}}  \def\r{\mathbf r}
\def\Z{\mathbb Z} \def\vrr{\varrho}
\def\L{\scr L}\def\Tt{\tt} \def\TT{\tt}\def\II{\mathbb I}
\def\i{{\rm in}}\def\Sect{{\rm Sect}}  \def\H{\mathbb H}
\def\M{\scr M}\def\Q{\mathbb Q} \def\texto{\text{o}}
\def\Rank{{\rm Rank}} \def\B{\scr B} \def\i{{\rm i}} \def\HR{\hat{\R}^d}
\def\to{\rightarrow}\def\l{\ell}
\def\8{\infty}\def\I{1}\def\U{\scr U} \def\n{{\mathbf n}}
\maketitle

\begin{abstract} The couplings by change of measure are applied to establish log-Harnack inequality(equivalently the entropy-cost estimate) for conditional McKean-Vlasov SDEs and derive the quantitative conditional propagation of chaos in relative entropy for mean field interacting particle system with common noise. For the log-Harnack inequality, two different types of couplings will be constructed for non-degenerate conditional McKean-Vlasov SDEs with multiplicative noise. As to the quantitative conditional propagation of chaos in relative entropy, the initial distribution of interacting particle system is allowed to be singular with that of limit equation. The above results are also extended to conditional distribution dependent stochastic Hamiltonian system.
 \end{abstract}

 \vskip 2cm

\section{Introduction}
Distribution dependent stochastic differential equations(SDEs) can be viewed as the limit equation of a single particle in the mean field interacting particle system as the number of particles goes to infinity, see \cite{SZ}. It is applied extensively in mean field games \cite{LL}. It is also called McKean-Vlasov SDE in the literature due to the work in \cite{McKean}. Different from the classical It\^{o}'s stochastic differential equation, the distribution of  McKean-Vlasov SDEs solves nonlinear Fokker-Planck-Kolmogorov equation. When there exists a common noise in the mean field interacting particle system, which is independent of the private noise of all particles, the limit equation of a single particle turns into a conditional distribution dependent SDE, which is called conditional McKean-Vlasov SDE, see \cite{CD}. Moreover, the conditional distribution of the solution with respect to the common noise is a probability measure-valued stochastic process, which solves stochastic nonlinear Fokker-Planck-Kolmogorov equation, see for instance \cite{CD,DSZ}. Compared with the McKean-Vlasov SDEs, there are fewer results on conditional  McKean-Vlasov ones, one can refer to \cite{BSW,BCEH,BLY,CD,CG,HSS,KX,DSZ,SW,STW,W21} for well-posedness, \cite{CG,KX,KT} for the study of stochastic nonlinear Fokker-Planck-Kolmogorov equation and \cite{BSW,BCEH,CFl,ELL,KT,VHP,SW,STW} for conditional propagation of chaos.

Let $\scr P(\R^d)$ be the space  of all  probability measures on $\R^d$ equipped with the weak topology.
For $k\geq 1$, let
$$\scr P_k(\R^d):=\big\{\mu\in \scr P(\R^d):\ \|\mu\|_k:= \mu(|\cdot|^k)^{\ff 1 k}<\infty\big\},$$
which is a Polish space under the $L^k$-Wasserstein distance
$$\W_k(\mu,\nu)= \inf_{\pi\in \C(\mu,\nu)} \bigg(\int_{\R^{d}\times\R^{d}} |x-y|^k \pi(\d x,\d y)\bigg)^{\ff 1 {k}},\ \  \mu,\nu\in \scr P_k(\R^d), $$ where $\C(\mu,\nu)$ is the set of all couplings of $\mu$ and $\nu$.
The relative entropy of two probability measures is defined as
$$\mathrm{Ent}(\nu|\mu)=\left\{
  \begin{array}{ll}
    \nu(\log(\frac{\d \nu}{\d \mu})), & \hbox{$\nu\ll\mu$;} \\
    \infty, & \hbox{otherwise.}
  \end{array}
\right.$$
Fix $T>0$. As in \cite[Section 2.1.3]{CD}, let $(\Omega^i, \scr F^i, (\scr F^i_t)_{t\geq 0},\P^i), i=0,1$ be two complete filtration probability spaces and $(\Omega, \scr F, (\scr F_t)_{t\geq 0},\P)$ be the completion of the product structure generated by them, i.e. $\Omega=\Omega^0\times \Omega^1$, $\scr F$ and $(\scr F_t)_{t\geq 0}$ are the completions of $\F^0\otimes\F^1$ and $(\scr F^0_t\otimes\scr F^1_t)_{t\geq 0}$ with respect to the product measure $\P=\P^0\times\P^1$. $W_t$ is the $d_W$-dimensional Brownian motion on $(\Omega^1, \scr F^1, (\scr F^1_t)_{t\geq 0},\P^1)$ while $B_t$ is the $d_B$-dimensional Brownian motions on $(\Omega^0, \scr F^0, (\scr F^0_t)_{t\geq 0},\P^0)$. Let $\F_t^B=\sigma\{B_s, s\in[0,t]\}\otimes\{\emptyset,\Omega^1\}, t\in[0,T]$. Consider conditional McKean-Vlasov SDEs:
\begin{align}\label{E0}\d  X_t=b_t(X_t,\L_{X_t|\F_t^B})\mathrm{d} t+\sigma_t(X_t,\L_{X_t|\F_t^B})\mathrm{d} W_t+\tilde{\sigma}_t(X_t,\L_{X_t|\F_t^B})\d B_t,
\end{align}
where $\L_{X_t|\F_t^B}$ stands for the regular conditional distribution of $X_t$ with respect to $\F_t^B$, $b:[0,T]\times \R^d\times\scr P(\R^d)\to\R^d$, $\sigma:[0,T]\times \R^d\times\scr P(\R^d)\to\R^d\otimes\R^{d_W}$ and $\tilde{\sigma}:[0,T]\times \R^d\times\scr P(\R^d)\to\R^d\otimes\R^{d_B}$ are measurable and  bounded on bounded set. The noise $B_t$ is usually called the common noise while $W_t$ is the private noise. Throughout the paper, we assume that the initial value $X_0$ is $(\Omega^1, \scr F_0^1)$-measurable. Note that when \eqref{E0} is well-posed, see Definition \ref{DE1} below, \cite[Proposition 2.9, Lemma 2.5]{CD} tells that $\{\L_{X_t|\F_t^B}\}_{t\in[0,T]}$ is a version of the $\{\mathcal{L}^1(X_t)\}_{t\in[0,T]}$ in \cite[(2.6)]{CD}. For more other assumptions on the initial value $X_0$, one can refer to \cite[Remark 2.10]{CD}.
\begin{defn}\label{DE1}
For any $\xi\in L^2(\Omega^1\to\R^d,\F^1_0,\P^1)$, we call a continuous and $\F_t$-adapted process $X_t$ with $\E\sup_{t\in[0,T]}|X_t|^2<\infty$ a solution to \eqref{E0} with initial value $\xi$, if $\L_{X_t|\F_t^B}$ is a continuous, $\F_t^B$-adapted and $\scr P_2(\R^d)$-valued process and it holds
$$\E\int_0^T\{|b_t(X_t,\L_{X_t|\F_t^B})|+|\sigma_t(X_t,\L_{X_t|\F_t^B})|^2+|\tilde{\sigma}_t(X_t,\L_{X_t|\F_t^B})|^2\}\d t<\infty$$
and $\P$-a.s.
$$X_s=\xi+\int_0^sb_t(X_t,\L_{X_t|\F_t^B})\mathrm{d} t+\int_0^s\sigma_t(X_t,\L_{X_t|\F_t^B})\mathrm{d} W_t+\int_0^s\tilde{\sigma}_t(X_t,\L_{X_t|\F_t^B})\d B_t,\ \ s\in[0,T].$$
We call \eqref{E0} is well-posed, if for any $\xi\in L^2(\Omega^1\to\R^d,\F^1_0,\P^1)$, it has a unique solution starting from $\xi$ which will be denoted by $X_t^\xi$ in the sequel.
\end{defn}
When \eqref{E0} has a solution, the conditional time-marginal distribution $\mu_t:=\L_{X_t|\F_t^B}$ solves measured-valued stochastic evolution equation, see \cite{CD,DSZ}, the study of which can be dated to \cite{D}. Since then, it has been intensively investigated. \cite{BCEH} derived the well-posedness of mean reflected forward and backward SDEs and obtained the propagation of chaos in Wasserstein distance for  associated interacting particle
system. Moreover, in the forward case, the conditional mean reflected SDEs and conditional propagation of chaos in Wasserstein distance are also studied; \cite{BLY} studied a systemic risk control problem by the central bank, which stabilizes the interbank system with borrowing and lending activities and the mean field optimal control is shown to satisfy a stochastic Fokker-Planck-Kolmogorov equation driven by the common noise; In \cite{KX}, the uniqueness for the stochastic nonlinear Fokker-Planck-Kolmogorov equation is proved in the class of solutions with squarely integrable density with respect to the Lebesgue measure; In \cite{CG}, the uniqueness is shown  by means of a duality argument to a backward stochastic PDE and \cite{DV} verifies the uniqueness of solutions by a dual method, coupling arguments as well as the Krylov-Rozovskii "variational" approach to SPDE.
In \cite{DSZ}, the  superposition principle and mimicking theorem for  conditional McKean-Vlasov SDE are derived, which establish the correspondence between conditional McKean-Vlasov SDE and stochastic nonlinear  Fokker-Planck-Kolmogorov equation under reasonable condition and also show that the conditional time-marginals of an It\^{o} process can be constructed by those of solution to a conditional McKean-Vlasov SDE with Markovian coefficients. This provides a probability method to investigate stochastic nonlinear  Fokker-Planck-Kolmogorov equation.

In recent years, the study of \eqref{E0} attracts much attention. \cite{HSS} proved that \eqref{E0} is well-posed if \eqref{E0} with $b=0$ is well-posed, $\sigma$ and $\tilde{\sigma}$ are distribution free and
$\sigma^{-1}b$ is bounded and Lipchitz continuous under total variation distance.
In \cite{KT}, the quantitative conditional propagation of chaos in weak convergence is provided, where $\sigma$ and $\tilde{\sigma}$ are distribution free and all the coefficients are regular enough in spatial-measure arguments.
\cite{CFl} proved conditional propagation of chaos in Wasserstein distance when $\sigma=0$, $b(x,\mu)=\mu(f(x-\cdot))$ for some Lipschitz function $f$. \cite{ELL} investigated conditional propagation of chaos for one dimensional SDEs driven by Poisson random measure and common Brownian motion noise, where $\tilde{\sigma}=\sqrt{\mu(f)}$ for some positive Lipschitz function $f$. For moderately interacting particle systems with environmental noise and singular interaction kernel such as the Biot-Savart and repulsive Poisson kernels, \cite{GL} proved that the mollified empirical measures converge in strong
norms to the unique (local) solutions of nonlinear Fokker-Planck-Kolmogorov equations. \cite{W21} studied the well-posedness in the case $\sigma=0$ by constructing image dependent SDE.
In \cite{VHP}, the quantitative conditional propagation of chaos in the sense of Wasserstein distance is studied for  stochastic spatial epidemic model, where the evolution of infection states are driven by the Poisson point processes and the displacement of individuals contains a common noise. One can also refer to \cite{SW, STW} for the (conditional)propagation of chaos for (conditional)McKean-Vlasov SDEs with regime-switching.

The propagation of chaos is a hot topic in the McKean-Vlasov frame($\tilde{\sigma}=0$). The quantitative propagation of chaos in strong sense is studied in \cite{SZ} by using synchronous coupling argument, where the coefficients are assumed to be Lipschitz continuous and the initial value of interacting particle system coincides with that of the limit equation. \cite{BJW,JW,JW1} apply the entropy method to derive the quantitative propagation of chaos in relative entropy with additive noise and singular interaction, for which the initial distribution of interacting particle system is assumed to be absolutely continuous with that of limit equation. In \cite{L21}, the authors give the sharp rate of propagation of chaos for some models such as bounded or uniformly continuous
interaction by BBGKY hierarchy. We should also mention that in \cite{MD},  the (uniform in time)quantitative propagation of chaos for genetic-type interacting particle system approximating model in the sense of relative entropy as well as $\mathbb{L}_\alpha(\alpha\in[1,\infty])$ estimate  and thus in the sense of total variation distance are obtained.

 As far as we know, the regularity estimate of conditional McKean-Vlasov SDEs with respect to the initial value such as the entropy-cost estimate is still open. In this paper, we try to construct the coupling by change of measure for conditional McKean-Vlasov SDEs \eqref{E0}. We will present two different couplings in the case with non-degenerate and multiplicative noise to derive the log-Harnack inequality, which is equivalent to entropy-cost estimate. In the distribution independent case, the log-Harnack inequality associated to a Markov semigroup $P_t$ is formulated as
\begin{equation*} \begin{split}P_t\log f(x)&\le \log P_tf(y)+ c(t) |x-y|^2, f\in \B_b^+(\R^d), t\in (0,T], x,y\in\R^d\end{split}\end{equation*}
for some nonnegative function $c$ with $\lim_{t\to 0}c(t)=\infty$, which implies the gradient-$L^2$ estimate:
$$|\nabla P_tf|^2\leq c(t)P_t|f|^2,\ \ t\in (0,T].$$
In the case of non-degenerate diffusion, it is also equivalent to the gradient-gradient estimate:
$$|\nabla P_tf|^2\leq CP_t|\nabla f|^2,\ \ t\in [0,T]$$
for some constant $C>0$.
One can refer to \cite[Chapter 1]{Wbook} for more details.

 Different from the McKean-Vlasov frame, the conditional distribution with respect to the common noise is a functional of common noise so that we have to overcome essential difficulties produced by this crucial difference. For instance, in the procedure of constructing coupling processes, we usually view the conditional distribution with respect to the common noise as a known functional of common noise so that the common noise need also be fixed. Hence, we can only construct a new private noise in coupling process. Moreover, since the private noise and the common noise are independent, when we calculate the expectation for a functional of $(W,B)$, we can firstly take conditional expectation with respect to the common noise in which the common noise can be viewed as a constant and then use the tower property of conditional expectation to realize this goal.

 We will also investigate the quantitative conditional propagation of chaos in the sense of Wasserstein distance, which together with coupling by change of measure implies the quantitative conditional propagation of chaos in relative entropy. Different from \cite{BJW,JW,JW1}, the initial distribution of interacting particle system is allowed to be singular with that of the limit equation. The main tool is an entropy inequality in \cite[Lemma 2.1]{23RW} as well as Wang's Harnack inequality with power. Furthermore, the associated assertions are derived by the method of coupling by change of measure for the conditional distribution dependent stochastic Hamiltonian system and mean field interacting stochastic Hamiltonian system with common noise.

When the conditional distribution is involved, an inequality is often used:
$$\E\W_2(\L_{\xi|\scr G}, \L_{\eta|\scr G})^2\leq \E\left\{\E(|\xi-\eta|^2|\scr G)\right\} =\E|\xi-\eta|^2$$
for any random variables $\xi,\eta$ with finite second monents and any sub-$\sigma$-algebra $\scr G\subset \F$.
Using Banach's fixed point theorem and repeating the proof of \cite[Proposition 2.11]{CD}, it is standard to obtain Lemma \ref{EEB} below under the following monotonicity condition {\bf(H)}. To save space, we omit the proof. One can also refer to \cite{W21} for the well-posedness under the monotonicity condition of image-dependent SDE, a type of special conditional McKean-Vlasov SDE aforementioned.
When \eqref{E0} is well-posed and for any $\gamma\in\scr P_2(\R^d)$ and any $\xi,\tilde{\xi}\in L^2(\Omega^1\to\R^d,\F^1_0,\P^1)$ with $\L_{\xi}=\L_{\tilde{\xi}}=\gamma$, it holds $\L_{X_t^\xi}=\L_{X_t^{\tilde{\xi}}}$, then we denote $P_t^\ast\gamma=\L_{X_t^\xi}$ and
$$P_tf(\gamma):=\int_{\R^d}f(x)( P_t^\ast\gamma)(\d x),\ \ f\in\scr B_b(\R^d).$$
\begin{enumerate} \item[{\bf (H)}] For any $t\in[0,T]$, $b_t, \sigma_t,\tilde{\sigma}_t$ are continuous in $\R^d\times\scr P_2(\R^d)$. There exists a constant $K\geq 0$  such that
 \begin{align*}
 &\|\sigma_t(x,\mu)-\sigma_t(y,\nu)\|_{HS}^2+\|\tt\si_t (x,\mu)-\tt\si_t (y,\nu)\|_{HS}^2
+2\<b_t(x,\mu)-b_t(y,\nu),x-y\>\\
& \le  K (|x-y|^2+\W_2(\mu,\nu)^2),\ \ t\in [0,T],\ x,y\in\R^d,\ \mu,\nu\in \scr P_2(\R^d).
\end{align*}
\end{enumerate}

\begin{lem}\label{EEB} Assume {\bf (H)}. Then \eqref{E0} is well-posed and  $\L_{X_t^\xi|\F_t^B}=\L_{X_t^{\tilde{\xi}}|\F_t^B}$ for any initial values $\xi,\tilde{\xi}\in L^2(\Omega^1\to\R^d,\F^1_0,\P^1)$ with $\L_{\xi}=\L_{\tilde{\xi}}$. Moreover, there exists a constant $C>0$ such that
\begin{align*}&\E\W_2(\L_{X_s^\xi|\F_s^B}, \L_{X_s^{\tilde{\xi}}|\F_s^B})^2+
\W_2(\L_{X_s^\xi}, \L_{X_s^{\tilde{\xi}}})^2\\
&\qquad\qquad\quad\leq C\W_2(\L_{\xi},\L_{\tilde{\xi}})^2,\ \ s\in[0,T],\ \ \xi,\tilde{\xi}\in L^2(\Omega^1\to\R^d,\F^1_0,\P^1).
\end{align*}
\end{lem}
When there are different probability measures on $(\Omega,\F)$, we use $\L^{\P}_{\xi}$ and $\L^{\P}_{\xi|\scr G}$ denote the distribution and regular conditional distribution of random variable $\xi$ with respect to sub-$\sigma$-algebra $\scr G\subset \F$ respectively under probability measure $\P$.

The remaining of the paper is organized as follows: In section 2,  we establish the log-Harnack inequality and thus the entropy-cost estimate for conditional McKean-Vlasov SDEs with non-degenerate and multiplicative noise and two different cases are considered. Moreover, we investigate the quantitative conditional propagation of chaos in Wasserstein distance and relative entropy. The corresponding results are derived in Section 3 for conditional distribution dependent  stochastic Hamiltonian system and mean field interacting stochastic Hamiltonian system with common noise.
\section{Non-degenerate case}
\subsection{Log-Harnack inequality}

To apply the coupling by change of measure to establish the log-Harnack inequality for conditional McKean-Vlasov SDEs, we assume $\sigma$ is distribution free and consider
\begin{align}\label{E1}\d  X_t=b_t(X_t,\L_{X_t|\F_t^B})\mathrm{d} t+\sigma_t(X_t)\mathrm{d} W_t+\tilde{\sigma}_t(X_t,\L_{X_t|\F_t^B})\d B_t.
\end{align}
In the following, we will investigate two different cases and construct corresponding coupling by change of measure to derive the log-Harnack inequality for \eqref{E1}.
\subsubsection{Case 1: $\tilde{\sigma}_t(x,\mu)=\tilde{\sigma}_t(x)$}
\begin{enumerate} \item[{\bf (A)}] For any $t\in[0,T],x\in\R^d$, $(\sigma_t\sigma^{\ast}_t)(x)$  is invertible and $b_t$ is continuous in $\R^d\times\scr P_2(\R^d)$. There exist $\lambda \in(0,1]$ and $K,\tilde{K}\geq 0$  such that
 \begin{align*}
 &\lambda^{-1}\geq(\sigma_t\sigma^{\ast}_t)(x)\geq \lambda,\ \ \|\sigma_t(x)-\sigma_t(y)\|_{HS}^2 \leq K|x-y|^2,\\
&\<b_t(x,\mu)-b_t(y,\nu),x-y\>\le  K (|x-y|\W_2(\mu, \nu)+|x-y|^2),\\
 &\|\tt\si_t (x)-\tt\si_t (y)\|_{HS}^2\leq \tilde{K} |x-y|^2,\ \ t\in [0,T],\ x,y\in\R^d,\ \mu,\nu\in \scr P_2(\R^d).
\end{align*}
\end{enumerate}
\begin{thm}\label{LHA} Assume {\bf (A)}.
Then there exists a constant $c>0$ such that
\begin{equation*}\begin{split}
&P_t\log f(\nu_0) \le \log P_t f(\mu_0)+c\left\{\frac{(3K+\tilde{K})}{1-\e^{-(3K+\tilde{K})t}}+ t\right\}\W_2(\mu_0,\nu_0)^2,\\
&\qquad\qquad\quad \quad\quad\qquad\qquad0<f\in \B_b(\R^d), \mu_0,\nu_0 \in \scr P_2(\R^d), t\in (0,T].\end{split} \end{equation*}
 \end{thm}
\begin{proof}
Let $X_0^{\mu_0},X_0^{\nu_0}$ be $(\Omega^1,\F^1_0)$-measurable such that
 \beq\label{COU} \L_{X^{\mu_0}_0}=\mu_0,\ \ \L_{X_0^{\nu_0}}=\nu_0,\ \ \E|X_0^{\mu_0}-X_0^{\nu_0}|^2=\W_{2}(\mu_0,\nu_0)^2.\end{equation}
 Let $X_t^{\mu_0}$ and $X_t^{\nu_0}$ solve \eqref{E1} with initial values $X_0^{\mu_0}$ and $X_0^{\nu_0}$ respectively. Denote
 \begin{align}\label{mnu}\nu_t=\L^{\P}_{{X^{\nu_0}_t|\F_t^B}},\ \ \mu_t=\L^{\P}_{{X^{\mu_0}_t|\F_t^B}},\ \ t\in[0,T].
 \end{align}
 Then it holds
 \beq\label{CDM} \begin{split}\d X_t^{\mu_0} &= b_t(X_t^{\mu_0}, \mu_t)\d t + \sigma_t(X_t^{\mu_0}) \d W_t+\tilde{\sigma}_t(X_t^{\mu_0}) \d B_t,\ \ t\in [0,T].
\end{split}\end{equation}
Let $t_0\in(0,T]$ and $\xi_t=\frac{1}{(3K+\tilde{K})}(1-\e^{(3K+\tilde{K})(t-t_0)})$, which satisfies
\begin{align}\label{XXY}-\xi_t'+(3K+\tilde{K})\xi_t=1.
\end{align}
Consider the following SDE:
\beq\label{YPR} \beg{split}&\d Y_t = b_t(Y_t, \nu_t)\d t + \sigma_t(Y_t) \d W_t+\tilde{\sigma}_t(Y_t) \d B_t\\
&\qquad\quad+\sigma_t(Y_t)[\sigma_t^\ast(\sigma_t\sigma_t^\ast)^{-1}] (X_t^{\mu_0})\frac{X_t^{\mu_0}-Y_t}{\xi_t}\d t,\quad \ t\in [0,t_0),\ Y_0=X_0^{\nu_0}.\end{split} \end{equation}
Let $\tau_n=t_0\wedge\inf\{t\in[0,t_0), |X_t^{\mu_0}|\vee|Y_t|\geq n\}$. Then $\P$-a.s. $\tau_n\uparrow t_0$ as $n\uparrow\infty$.
Let
\beq\label{gga}\begin{split} &\gamma_t:=[\sigma_t^\ast(\sigma_t\sigma_t^\ast)^{-1}](X_t^{\mu_0}) \frac{Y_t-X_t^{\mu_0}}{\xi_t},\\
&\hat{W}_t :=   W_t-\int_0^t \gamma_s\d s,\ \ R_t:= \e^{\int_0^{t}\<\gamma_r, \d W _r\> -\ff 1 2 \int_0^{t} |\gamma_r|^2\d r},\\
& \Q_t:= R_t\P, \ \ t\in [0,t_0).
\end{split}\end{equation}
Fix $s\in[0,t_0)$.
According to Girsanov's theorem,  under the weighted probability $\Q_{s\wedge\tau_n},$
$(\hat{W}_{t},B_{t})$ is a $(d_{W}+d_{B})$-dimensional Brownian motion up to time $s\wedge \tau_n$.

Then \eqref{CDM} and \eqref{YPR} can be rewritten as
\begin{equation*} \begin{split}\d X_t^{\mu_0}&= b_t(X_t^{\mu_0}, \mu_t)\d t + \sigma_t(X_t^{\mu_0}) \d \hat{W}_t+\tilde{\sigma}_t(X_t^{\mu_0}) \d B_t+\frac{Y_t-X_t^{\mu_0}}{\xi_t},\ \ t\in [0,s\wedge \tau_n],
\end{split}\end{equation*}
and
 \begin{equation*} \beg{split}&\d Y_t= b_t(Y_t, \nu_t)\d t + \sigma_t(Y_t) \d \hat{W}_t+\tilde{\sigma}_t(Y_t) \d B_t,\quad \ t\in [0,s\wedge \tau_n],\ \ Y_0=X_0^{\nu_0}.\end{split} \end{equation*}
It follows from  It\^{o}'s formula that
\beq\label{XYX}\begin{split} &\d \frac{|Y_t-X_t^{\mu_0}|^2}{\xi_t}\\
&=-\frac{\xi_t'|Y_t-X_t^{\mu_0}|^2}{\xi_t^2}\d t+\frac{2\< b_t(Y_t, \nu_t)- b_t(X_t^{\mu_0}, \mu_t),Y_t-X_t^{\mu_0}\>}{\xi_t}\d t-2\frac{|Y_t-X_t^{\mu_0}|^2}{\xi_t^2} \d t\\
&+\frac{2\<[\sigma_t(Y_t)-\sigma_t(X_t^{\mu_0})]\d \hat{W}_t+[\tilde{\sigma}_t(Y_t)-\tilde{\sigma}_t(X_t^{\mu_0})]\d B_t, Y_t-X_t^{\mu_0}\>}{\xi_t}\\
&+\frac{\|\sigma_t(Y_t)-\sigma_t(X_t^{\mu_0})\|_{HS}^2+ \|\tilde{\sigma}_t(Y_t)-\tilde{\sigma}_t(X_t^{\mu_0})\|_{HS}^2}{\xi_t}\d t,\ \ t\in [0,s\wedge \tau_n].
\end{split}\end{equation}
In view of {\bf(A)}, we conclude that
\begin{align*}
&\frac{2\< b_t(Y_t, \nu_t)- b_t(X_t^{\mu_0}, \mu_t),Y_t-X_t^{\mu_0}\>}{\xi_t}\\
&\leq \frac{2K|Y_t-X_t^{\mu_0}|\W_2(\mu_t,\nu_t)+2K|Y_t-X_t^{\mu_0}|^2}{\xi_t}\\
&\leq \frac{1}{2}\frac{|Y_t-X_t^{\mu_0}|^2}{\xi_t^2}+2K^2\W_2(\mu_t,\nu_t)^2+\frac{2K\xi_t|Y_t-X_t^{\mu_0}|^2}{\xi_t^2},
\end{align*}
and
\begin{align*}
&\frac{\|\sigma_t(Y_t)-\sigma_t(X_t^{\mu_0})\|_{HS}^2+\|\tilde{\sigma}_t(Y_t)-\tilde{\sigma}_t(X_t^{\mu_0})\|_{HS}^2}{\xi_t}\leq \frac{(K+\tilde{K})\xi_t|Y_t-X_t^{\mu_0}|^2}{\xi_t^2}.
\end{align*}
This together with \eqref{XYX} gives
\beq\label{XYI}\begin{split} \d \frac{|Y_t-X_t^{\mu_0}|^2}{\xi_t}&\leq\frac{[-\xi_t'+(3K+\tilde{K})\xi_t-\frac{3}{2}]|Y_t-X_t^{\mu_0}|^2}{\xi_t^2}\d t\\
&+2K^2\W_2(\mu_t,\nu_t)^2\d t+\d M_t,\ \ t\in [0,s\wedge \tau_n],
\end{split}\end{equation}
where
$$\d M_t=\frac{2\<[\sigma_t(Y_t)-\sigma_t(X_t^{\mu_0})]\d \hat{W}_t+[\tilde{\sigma}_t(Y_t)-\tilde{\sigma}_t(X_t^{\mu_0})]\d B_t, Y_t-X_t^{\mu_0}\>}{\xi_t}.$$
Since $W$ is independent of $B$, we have
$$\E(R_{s\wedge\tau_n}|\F_s^B)=1,$$
which together with \eqref{mnu}, the definition of $\mu_t,\nu_t$ and Lemma \ref{EEB} implies
\begin{align*}&\E_{\Q_{s\wedge\tau_n}}\int_0^{s\wedge\tau_n}\W_2(\mu_t,\nu_t)^2\d t\\
&\leq \E\left\{\E(R_{s\wedge\tau_n}|\F_s^B)\int_0^{s}\W_2(\mu_t,\nu_t)^2\d t\right\}\\
&=\E\int_0^{s}\W_2(\mu_t,\nu_t)^2\d t\leq Cs\W_2(\mu_0,\nu_0)^2.
\end{align*}
Combing this with \eqref{XXY} and \eqref{XYI}, we derive
\begin{align*}
&\E_{\Q_{s\wedge\tau_n}}\int_0^{s\wedge\tau_n}\frac{|Y_t-X_t^{\mu_0}|^2}{\xi_t^2}\d t\leq 2\E_{\Q_{s\wedge\tau_n}}\frac{|Y_0-X_0^{\mu_0}|^2}{\xi_0}+\E_{\Q_{s\wedge\tau_n}}\int_0^{s\wedge\tau_n}4K^2\W_2(\mu_t,\nu_t)^2\d t\\
&\leq 2\frac{\W_2(\mu_0,\nu_0)^2}{\xi_0}+ 4K^2Cs\W_2(\mu_0,\nu_0)^2.
\end{align*}
Hence, by \eqref{gga} and {\bf (A)}, we find a constant $c_1>0$ such that
\begin{equation*}\begin{split}  &\E[R_{s\wedge\tau_n}\log R_{s\wedge\tau_n}]=\E_{\Q_{s\wedge\tau_n}}[\log R_{s\wedge\tau_n}]=\frac{1}{2}\E_{\Q_{s\wedge\tau_n}}\int_0^{s\wedge \tau_n} |\gamma_t|^2\d t\\
&\leq c_1\frac{\W_2(\mu_0,\nu_0)^2}{\xi_0}+ c_1s\W_2(\mu_0,\nu_0)^2.\end{split}\end{equation*}
Consequently, $\{R_{s\wedge\tau_n}\}_{n\geq 1}$ is a uniformly integrable martingale under $\P$, which together with the martingale convergence theorem and Fatou's lemma implies that
\beq\label{REN}\begin{split}  &\frac{1}{2}\E_{\Q_{s}}\int_0^{s} |\gamma_t|^2\d t=\E[R_{s}\log R_{s}]\leq c_1\frac{\W_2(\mu_0,\nu_0)^2}{\xi_0}+ c_1s\W_2(\mu_0,\nu_0)^2,\ \ s\in[0,t_0).\end{split}\end{equation}
This means that $\{R_{s}\}_{s\in[0,t_0]}$ is a uniformly integrable martingale under $\P$ and Girsanov's theorem yields that under the weighted probability $\Q_{t_0},$
$(\hat{W}_{t},B_{t})_{t\in[0,t_0]}$ is a $(d_{W}+d_{B})$-dimensional Brownian motion. Moreover,
$\Q_{t_0}$-a.s. $Y_{t_0}=X_{t_0}^{\mu_0}$ by \eqref{REN} for $s=t_0$ due to Fatou's lemma.
On the other hand, consider the conditional McKean-Vlasov SDE
 \beq\label{YHB} \beg{split}&\d \tilde{Y}_t= b_t(\tilde{Y}_t, \L^{\Q_{t_0}}_{\tilde{Y}_t|\F_t^B})\d t + \sigma_t(\tilde{Y}_t) \d \hat{W}_t+\tilde{\sigma}_t(\tilde{Y}_t) \d B_t,\quad \ t\in [0,t_0],\ \ \tilde{Y}_0=X_0^{\nu_0}.\end{split} \end{equation}
According to \cite[Proposition 2.11]{CD}, we derive $\nu_t=\L^{\Q_{t_0}}_{\tilde{Y}_t|\F_t^B}$ and $\L^{\P}_{X^{\nu_0}_t}=\L^{\Q_{t_0}}_{\tilde{Y}_t}$ so that \eqref{YHB} can be rewritten as
 \beq\label{YHC} \beg{split}&\d \tilde{Y}_t= b_t(\tilde{Y}_t, \nu_t)\d t + \sigma_t(\tilde{Y}_t) \d \hat{W}_t+\tilde{\sigma}_t(\tilde{Y}_t) \d B_t,\quad \ t\in [0,t_0], \tilde{Y}_0=X_0^{\nu_0}.\end{split} \end{equation}
The strong uniqueness of \eqref{YHC} implies $Y_t=\tilde{Y}_t,t\in[0,t_0]$. In fact, \eqref{YHC} is an SDE with random coefficients, the well-posedness of which can be proved by standard argument under the assumption {\bf(A)}. Therefore, $\L^{\Q_{t_0}}_{Y_t|\F_t^B}=\L^{\P}_{{X^{\nu_0}_t|\F_t^B}}=\nu_t$ and $\L^{\Q_{t_0}}_{Y_t}=\L^{\P}_{X^{\nu_0}_t}$.
Combining this with Young's  inequality and \eqref{REN} for $s=t_0$ due to Fatou's lemma,  we derive
\beg{align*} &P_{t_0}\log  f(\nu_0)= \E_{\Q_{t_0}} [\log f(Y_{t_0}) ]
 = \E[ R_{t_0}  \log f(X_{t_0}^{\mu_0}) ] \\
 &\le \log \E [ f(X_{t_0}^{\mu_0})] + \E [R_{t_0}\log R_{t_0}] \\
& \leq \log P_{t_0}  f(\mu_0) +c_1\frac{\W_2(\mu_0,\nu_0)^2}{\xi_0}+ c_1t_0\W_2(\mu_0,\nu_0)^2,\ \ 0< f\in \B_b(\R^d).\end{align*}
 Therefore, we complete the proof by the definition of $\xi_0$.
 \end{proof}
\subsubsection{Case 2: $\tilde{\sigma}_t(x,\mu)=\tilde{\sigma}_t(\mu)$}
In the second case, we assume that $\tilde{\sigma}$ only depends on the time-distribution arguments, i.e. consider
  \beq\label{E'} \d X_t= b_t(X_t, \L_{X_t|\F_{t}^B})\d t+  \sigma_t(X_t) \d W_t+ \tt \si_t( \L_{X_t|\F_{t}^B})\d B_t,\ \ t\in [0,T]. \end{equation}
To establish the  log-Harnack inequality, we make   the following Lipschitz assumption instead of {\bf(A)}.
\begin{enumerate} \item[{\bf (B)}] $(\sigma_t\sigma^{\ast}_t)(x)$ is invertible and there exist $\lambda \in(0,1]$ and $K,\tilde{K}\geq 0$  such that
 \begin{align*}
 &\lambda^{-1}\geq(\sigma_t\sigma^{\ast}_t)(x)\geq \lambda,\ \ \|\sigma_t(x)-\sigma_t(y)\|_{HS}^2 \leq K|x-y|^2,\\
&|b_t(x,\mu)-b_t(y,\nu)| \le  K (|x-y|+\W_2(\mu, \nu)),\\
& \|\tt\si_t (\mu)-\tt\si_t (\nu)\|^2\leq \tilde{K} \W_2(\mu, \nu)^2,
\ \  \ t\in [0,T],\ x,y\in\R^d,\ \mu,\nu\in \scr P_2(\R^d).
\end{align*}
\end{enumerate}
Again according to Lemma \ref{EEB}, assumption {\bf (B)} implies Lemma \ref{EEB} holds for SDE \eqref{E'} replacing \eqref{E0}. In the case that $\tilde{\sigma}_t(x,\mu)=\tilde{\sigma}_t(\mu)$, the coupling used in Section 2.1 is unavailable so that we need to construct a new coupling by change of measure which involves in conditional probability with respect to $\F_t^B$.

\begin{thm}\label{Loh} Assume {\bf (B)}.
Then there exists a constant $c>0$ such that for any $0<f\in \B_b(\R^d), \mu_0,\nu_0 \in \scr P_2(\R^d), t\in (0,T]$ and $\xi,\tilde{\xi}\in L^2(\Omega^1\to\R^d,\F^1_0,\P^1)$ with $\L_{\xi}=\mu_0, \L_{\tilde{\xi}}=\nu_0$,
$$\E\{\mathrm{Ent}(\L_{X^{\xi}_t|\F_t^B}|\L_{X^{\tilde{\xi}}_t|\F_t^B})\}\le c\left\{\ff{4K}{1-\e^{-4K t_0}}+\int_0^{t_0}\frac{4Kt}{1-\e^{-4Kt}}\d t\right\}   \W_2(\mu_0,\nu_0)^2,$$
and consequently,
\begin{equation}\label{logHa}\begin{split}
&P_t\log f(\nu_0) \le \log P_t f(\mu_0)+c\left\{\ff{4K}{1-\e^{-4K t}}+\int_0^{t}\frac{4Kr}{1-\e^{-4Kr}}\d r\right\} \W_2(\mu_0,\nu_0)^2.
\end{split} \end{equation}
 \end{thm}
\begin{proof}
 For fixed $\mu_0,\nu_0\in \scr P_2(\R^d)$, let $X_0^{\mu_0},X_0^{\nu_0}$ be chosen in \eqref{COU}.
 Let $X_t^{\mu_0}$ and $X_t^{\nu_0}$ solve \eqref{E'} with initial value $X_0^{\mu_0}$ and $X_0^{\nu_0}$ respectively and $\mu_t$ and $\nu_t$ be defined in \eqref{mnu}. Define
   \begin{equation}\label{XM} \eta_t^\mu:= \int_0^t \tt \si_s(\mu_s)\d B_s,\ \ t\in [0,T].\end{equation}
By  {\bf (B)}, BDG's inequality and Lemma \ref{EEB},  we find   constants $C_1>0$ such that
\beq\label{-2} \begin{split}&\E\Big[\sup_{t\in [0,T]} |\eta_{t}^\mu-\eta_{t}^\nu|^2\Big]\le \tilde{K}^2\E\int_0^T\W_2(\mu_s,\nu_s)^2\d s\leq C_1 \tilde{K}^2T\W_2(\mu_0,\nu_0)^2, \end{split}\end{equation}
and
\beq\label{etn} \E |\eta_{t_1}^\mu-\eta_{t_2}^\mu-(\eta_{t_1}^\nu-\eta_{t_2}^\nu)|^2\le C_1\tilde{K}^2 \W_2(\mu_0,\nu_0)^2|t_1-t_2|,\ \ t_1,t_2\in[0,T]. \end{equation}
Moreover, we derive from \eqref{mnu} that
\begin{equation*} \d X_t^{\mu_0} = b_t(X_t^{\mu_0}, \mu_t)\d t + \sigma_t(X_t^{\mu_0}) \d W_t +\tt\si_t(\mu_t)\d B_t,\ \ t\in [0,T].\end{equation*}
Then
 $\hat X_t^{\mu_0}:= X_t^{\mu_0}-\eta_t^\mu,\ \ t\in [0,T]$ solves
\beq\label{hat} \d \hat{X}_t^{\mu_0} = b_t(\hat{X}_t^{\mu_0}+\eta_t^\mu, \mu_t)\d t + \sigma_t(\hat{X}_t^{\mu_0}+\eta_t^\mu) \d W_t,\ \ t\in [0,T].\end{equation}
Let $t_0\in(0,T]$ and $\xi_t=\frac{1}{4K}(1-\e^{4K(t-t_0)})$ and it holds
\begin{align}\label{xi1}-\xi_t'+4K\xi_t=1.
\end{align}
Now, we construct the coupling process:
\beq\label{2} \beg{split}\d \hat{Y}_t^\nu& = b_t(\hat{Y}_t^\nu+\eta_t^\nu, \nu_t)\d t + \sigma_t(\hat{Y}_t^\nu+\eta_t^\nu) \d W_t\\
&+\sigma_t(\hat{Y}_t^\nu+\eta_t^\nu)[\sigma_t^\ast(\sigma_t\sigma_t^\ast)^{-1}] (\hat{X}_t^{\mu_0}+\eta_t^\mu)\frac{(\hat{X}_t^{\mu_0}+\eta_{t_0}^\mu)-(\hat{Y}_t^\nu+\eta_{t_0}^\nu)}{\xi_t}\d t,\\
&\quad \ t\in [0,t_0),\ \hat{Y}_0^\nu=X_0^{\nu_0}.\end{split} \end{equation}
Define
\begin{align}\label{PTY}\P^{B}:= \P(\ \cdot\ |\F_T^B),\ \ \E^{B}:= \E(\ \cdot\ | \F_T^B).\end{align}
Set $\tau_n=t_0\wedge\inf\{t\in[0,t_0), |\hat{X}_t^\mu+\eta_{t_0}^\mu|\vee|\hat{Y}_t^\nu+\eta_{t_0}^\nu|\geq n\}$. Then we have $\P^{B}$-a.s. $\tau_n\uparrow t_0$ as $n\uparrow\infty$.
Let
\begin{align}\label{ETA}
\nonumber&\beta_t:=[\sigma_t^\ast(\sigma_t\sigma_t^\ast)^{-1}](\hat{X}_t^{\mu_0}+\eta_t^\mu) \frac{(\hat{Y}_t^\nu+\eta_{t_0}^\nu)-(\hat{X}_t^{\mu_0}+\eta_{t_0}^\mu)}{\xi_t},\\
&\tilde{W}_t :=   W_t-\int_0^t \beta_s\d s,\ \ R_t:= \e^{\int_0^{t}\<\beta_r, \d W _r\> -\ff 1 2 \int_0^{t} |\beta_r|^2\d r},\\
\nonumber& \Q_t^{B}:= R_t\P^{B}, \ \ t\in [0,t_0).
\end{align}
Fix $s\in[0,t_0)$.
Girsanov's theorem yields that under the weighted conditional probability $\Q_{s\wedge\tau_n}^{B},$
$\tilde{W}_{t}$ is a $d_W$-dimensional Brownian motion on $[0,s\wedge\tau_n]$.
Hence, \eqref{hat} and \eqref{2} can be reformulated as
\begin{equation*} \begin{split}\d [\hat{X}_t^{\mu_0} +\eta_{t_0}^\mu-\eta_{t_0}^\nu)]&= b_t(\hat{X}_t^{\mu_0}+\eta_t^\mu, \mu_t)\d t + \sigma_t(\hat{X}_t^{\mu_0}+\eta_t^\mu) \d \tilde{W}_t\\
&+\frac{(\hat{Y}_t^\nu+\eta_{t_0}^\nu)-(\hat{X}_t^{\mu_0}+\eta_{t_0}^\mu)}{\xi_t},\ \ t\in [0,s\wedge \tau_n],
\end{split}\end{equation*}
and
 \beq\label{212} \beg{split}&\d \hat{Y}_t^\nu = b_t(\hat{Y}_t^\nu+\eta_t^\nu,\nu_t)\d t + \sigma_t(\hat{Y}_t^\nu+\eta_t^\nu) \d \tilde{W}_t,\quad \ t\in [0,s\wedge \tau_n],\ \hat{Y}_0^\nu=X_0^{\nu_0}.\end{split} \end{equation}
 By    It\^{o}'s formula, we obtain
\beq\label{xi3}\begin{split} &\d \frac{|(\hat{Y}_t^\nu+\eta_{t_0}^\nu)-(\hat{X}_t^{\mu_0}+\eta_{t_0}^\mu)|^2}{\xi_t}\d t\\
&=-\frac{\xi_t'|(\hat{Y}_t^\nu+\eta_{t_0}^\nu)-(\hat{X}_t^{\mu_0}+\eta_{t_0}^\mu)|^2}{\xi_t^2}\d t\\
&+\frac{2\< b_t(\hat{Y}_t^\nu+\eta_t^\nu, \nu_t)- b_t(\hat{X}_t^{\mu_0}+\eta_t^\mu, \mu_t),(\hat{Y}_t^\nu+\eta_{t_0}^\nu)-(\hat{X}_t^{\mu_0}+\eta_{t_0}^\mu)\>}{\xi_t}\d t\\
&+\frac{2\<[\sigma_t(\hat{Y}_t^\nu+\eta_t^\nu)-\sigma_t(\hat{X}_t^{\mu_0}+\eta_t^\mu)]\d \tilde{W}_t, (\hat{Y}_t^\nu+\eta_{t_0}^\nu)-(\hat{X}_t^{\mu_0}+\eta_{t_0}^\mu)\>}{\xi_t}\\
&-2\frac{|(\hat{Y}_t^\nu+\eta_{t_0}^\nu)-(\hat{X}_t^{\mu_0}+\eta_{t_0}^\mu)|^2}{\xi_t^2} \d t\\ &+\frac{\|\sigma_t(\hat{Y}_t^\nu+\eta_t^\nu)-\sigma_t(\hat{X}_t^{\mu_0}+\eta_t^\mu)\|_{HS}^2}{\xi_t}\d t,\ \ t\in [0,s\wedge \tau_n].
\end{split}\end{equation}
{\bf(B)} implies that
\begin{align*}
&\frac{2\< b_t(\hat{Y}_t^\nu+\eta_t^\nu, \nu_t)- b_t(\hat{X}_t^{\mu_0}+\eta_t^\mu, \mu_t),(\hat{Y}_t^\nu+\eta_{t_0}^\nu)-(\hat{X}_t^{\mu_0}+\eta_{t_0}^\mu)\>}{\xi_t}\\
&\leq \frac{2K|(\hat{Y}_t^\nu+\eta_{t_0}^\nu)-(\hat{X}_t^{\mu_0}+\eta_{t_0}^\mu)|^2}{\xi_t}\\
&+\frac{2K[\W_2(\nu_t,\mu_t)+ |\eta_{t}^\nu-\eta_{t_0}^\nu-(\eta_{t}^\mu-\eta_{t_0}^\mu)|] |(\hat{Y}_t^\nu+\eta_{t_0}^\nu)-(\hat{X}_t^{\mu_0}+\eta_{t_0}^\mu)|}{\xi_t}\\
&\leq \frac{[2K\xi_t+\frac{1}{2}]|(\hat{Y}_t^\nu+\eta_{t_0}^\nu)-(\hat{X}_t^{\mu_0}+\eta_{t_0}^\mu)|^2}{\xi_t^2} +2K^2[\W_2(\nu_t,\mu_t)+ |\eta_{t}^\nu-\eta_{t_0}^\nu-(\eta_{t}^\mu-\eta_{t_0}^\mu)|] ^2,
\end{align*}
and
\begin{align*}
&\frac{\|\sigma_t(\hat{Y}_t^\nu+\eta_t^\nu)-\sigma_t(\hat{X}_t^{\mu_0}+\eta_t^\mu)\|_{HS}^2}{\xi_t}\\
&\leq \frac{2K\xi_t|(\hat{Y}_t^\nu+\eta_{t_0}^\nu)-(\hat{X}_t^{\mu_0}+\eta_{t_0}^\mu)|^2}{\xi_t^2}+ \frac{2K|\eta_{t}^\nu-\eta_{t_0}^\nu-(\eta_{t}^\mu-\eta_{t_0}^\mu)|^2}{\xi_t}.
\end{align*}
This together with \eqref{xi3} yields that
\begin{equation*}\begin{split} &\d \frac{|(\hat{Y}_t^\nu+\eta_{t_0}^\nu)-(\hat{X}_t^{\mu_0}+\eta_{t_0}^\mu)|^2}{\xi_t}\\
&\leq \frac{[-\xi_t'+4K\xi_t-\frac{3}{2}]|(\hat{Y}_t^\nu+\eta_{t_0}^\nu)-(\hat{X}_t^{\mu_0}+\eta_{t_0}^\mu)|^2}{\xi_t^2}\d t\\
&+2K^2[\W_2(\nu_t,\mu_t)+ |\eta_{t}^\nu-\eta_{t_0}^\nu-(\eta_{t}^\mu-\eta_{t_0}^\mu)|] ^2\d t\\
&+\frac{2K|\eta_{t}^\nu-\eta_{t_0}^\nu-(\eta_{t}^\mu-\eta_{t_0}^\mu)|^2}{\xi_t}\d t\\
&+\frac{2\<[\sigma_t(\hat{Y}_t^\nu+\eta_t^\nu)-\sigma_t(X_t^\mu+\eta_t^\mu)]\d \tilde{W}_t, (\hat{Y}_t^\nu+\eta_{t_0}^\nu)-(\hat{X}_t^\mu+\eta_{t_0}^\mu)\>}{\xi_t},\ \ t\in [0,s\wedge \tau_n].
\end{split}\end{equation*}
Combining this with \eqref{xi1}, we deduce
\begin{align*}
&\E_{\Q_{s\wedge\tau_n}^{B}}\int_0^{s\wedge\tau_n}\frac{|(\hat{Y}_t^\nu+\eta_{t_0}^\nu)-(\hat{X}_t^{\mu_0}+\eta_{t_0}^\mu)|^2}{\xi_t^2}\d t\\
&\leq \frac{2\E^B|(\hat{Y}_0^\nu+\eta_{t_0}^\nu)-(\hat{X}_0^{\mu_0}+\eta_{t_0}^\mu)|^2}{\xi_0}+8K^2\int_0^s\W_2(\nu_t,\mu_t)^2\d t\\
&+ \int_0^{s}8K^2|\eta_{t}^\nu-\eta_{t_0}^\nu-(\eta_{t}^\mu-\eta_{t_0}^\mu)| ^2\d t+\int_0^{s}\frac{4K|\eta_{t}^\nu-\eta_{t_0}^\nu-(\eta_{t}^\mu-\eta_{t_0}^\mu)|^2}{\xi_t}\d t.
\end{align*}
As a result, it follows from \eqref{ETA} and {\bf (B)} that
\begin{equation*}\begin{split}  &\E^{B}[R_{s\wedge\tau_n}\log R_{s\wedge\tau_n}]=\E_{\Q_{s\wedge\tau_n}^{B}}[\log R_{s\wedge\tau_n}]=\frac{1}{2}\E_{\Q_{s\wedge\tau_n}^{B}}\int_0^{s\wedge \tau_n} |\beta_t|^2\d t\\
&\leq \lambda^{-1}\frac{\E^B|(\hat{Y}_0^\nu+\eta_{t_0}^\nu)-(\hat{X}_0^\mu+\eta_{t_0}^\mu)|^2}{\xi_0}+4\lambda^{-1}K^2\int_0^s\W_2(\nu_t,\mu_t)^2\d t\\
&+ 4\lambda^{-1}K^2\int_0^{s}|\eta_{t}^\nu-\eta_{t_0}^\nu-(\eta_{t}^\mu-\eta_{t_0}^\mu)| ^2\d t+2\lambda^{-1}K\int_0^{s}\frac{|\eta_{t}^\nu-\eta_{t_0}^\nu-(\eta_{t}^\mu-\eta_{t_0}^\mu)|^2}{\xi_t}\d t.\end{split}\end{equation*}
This means that $\{R_{s\wedge\tau_n}\}_{n\geq 1}$ is a uniform integrable martingale under $\P^{B}$ so that we derive from the martingale convergence theorem and Fatou's lemma that
\begin{align}\label{*N3}
\nonumber &\frac{1}{2}\E_{\Q_{s}^{B}}\int_0^{s} |\beta_t|^2\d t=\E^{B}[R_{s}\log R_{s}]\\
&\leq \lambda^{-1}\frac{\E^B|(\hat{Y}_0^\nu+\eta_{t_0}^\nu)-(\hat{X}_0^\mu+\eta_{t_0}^\mu)|^2}{\xi_0}+4\lambda^{-1}K^2\int_0^s\W_2(\nu_t,\mu_t)^2\d t\\
\nonumber&+ 4\lambda^{-1}K^2\int_0^{s}|\eta_{t}^\nu-\eta_{t_0}^\nu-(\eta_{t}^\mu-\eta_{t_0}^\mu)| ^2\d t\\
\nonumber&+2\lambda^{-1}K\int_0^{s}\frac{|\eta_{t}^\nu-\eta_{t_0}^\nu-(\eta_{t}^\mu-\eta_{t_0}^\mu)|^2}{\xi_t}\d t,\ \ s\in[0,t_0).\end{align}
This combined with \eqref{etn} implies that $\{R_{s}\}_{s\in[0,t_0]}$ is a uniform integrable martingale under $\P^{B}$ and $\Q_{t_0}^{B}$-a.s. $\hat{Y}_{t_0}^\nu+\eta_{t_0}^\nu=\hat{X}_{t_0}^{\mu_0}+\eta_{t_0}^\mu$ in view of the definition of $\beta_t$ and $\xi_t$ and \eqref{*N3} for $s=t_0$ due to the Fatou's lemma. It follows from the weak uniqueness to \eqref{212} on $[0,t_0]$ due to {\bf(B)} that
\begin{align}\label{HAY}\L^{\Q_{t_0}^{B}}_{\hat Y_{t}^\nu } =\L^{\P^{B}}_{\hat X_{t}^{\nu_0}},\ \ t\in[0,t_0],
 \end{align}where $\hat X_{t}^{\nu_0}=X_{t}^{\nu_0}-\eta^\nu_t, t\in[0,T]$ solves
$$ \d \hat{X}_t^{\nu_0} = b_t(\hat{X}_t^{\nu_0}+\eta_t^\nu, \nu_t)\d t + \sigma_t(\hat{X}_t^{\nu_0}+\eta_t^\nu) \d W_t,\ \ t\in [0,T].$$
 Let   $Y_{t}^\nu:=\hat{Y}_t^\nu+\eta_t^\nu, t\in[0,t_0]$. Since
 $\eta_{t_0}^\nu$ is measurable with respect to $\F_{t_0}^B$, it follows from \eqref{HAY} that
\begin{align}\label{HHA}\L^{\Q_{t_0}^{B}}_{Y_{t_0}^\nu}= \L^{\Q_{t_0}^{B}}_{\hat Y_{t_0} +\eta_{t_0}^\nu} = \L^{\P^{B}}_{\hat X_{t_0}^{\nu_0} +\eta_{t_0}^\nu} = \L^{\P^{B}}_{X_{t_0}^{\nu_0} }.
\end{align}
Again by \eqref{*N3} for $s=t_0$, \eqref{HHA}, the fact $\Q_{t_0}^{B}$-a.s.  $Y_{t_0}^\nu=\hat{Y}_{t_0}^\nu+\eta_{t_0}^\nu=\hat{X}_{t_0}^{\mu_0}+\eta_{t_0}^\mu=X_{t_0}^{\mu_0}$ and Young's  inequality, we conclude that
\beg{align*} &\E^{B}[ \log f(X_{t_0}^{\nu_0})] = \E_{\Q_{t_0}^{B}} [\log f(Y_{t_0}^\nu) ]\\
 &= \E^{B}[ R_{t_0}  \log f(X_{t_0}^{\mu_0}) ]  \le \log \E^{B} [ f(X_{t_0}^{\mu_0})] + \E^{B} [R_{t_0}\log R_{t_0}] \\
& \leq \log \E^{B} [ f(X_{t_0}^{\mu_0})]  +\lambda^{-1}\frac{\E^B|(\hat{Y}_0^\nu+\eta_{t_0}^\nu)-(\hat{X}_0^\mu+\eta_{t_0}^\mu)|^2}{\xi_0}+4\lambda^{-1}K^2\int_0^{t_0}\W_2(\nu_t,\mu_t)^2\d t\\
&+ 4\lambda^{-1}K^2\int_0^{t_0}|\eta_{t}^\nu-\eta_{t_0}^\nu-(\eta_{t}^\mu-\eta_{t_0}^\mu)| ^2\d t\\
&+2\lambda^{-1}K\int_0^{t_0}\frac{|\eta_{t}^\nu-\eta_{t_0}^\nu-(\eta_{t}^\mu-\eta_{t_0}^\mu)|^2}{\xi_t}\d t=:\log \E^{B} [ f(X_{t_0}^{\mu_0})]+\Phi_{t_0},\ \ 0< f\in \B_b(\R^d).\end{align*}
 So,
 by \cite[Theorem 1.4.2(2)]{Wbook}, Lemma \ref{EEB}, \eqref{-2}, \eqref{etn} and \eqref{ETA},    we find a constant $c>0$ such that
\begin{align*}\E\{\mathrm{Ent}(\L_{X^{\nu_0}_t|\F_t^B}|\L_{X^{\mu_0}_t|\F_t^B})\}&\leq \E \Phi_{t_0}\le c\left\{\ff{4K}{1-\e^{-4K t_0}}+\int_0^{t_0}\frac{4Kt}{1-\e^{-4Kt}}\d t\right\}   \W_2(\mu_0,\nu_0)^2.
\end{align*}
This together with the fact
\begin{align}\label{eng}\mathrm{Ent}(\L_{X^{\nu_0}_t}|\L_{X^{\mu_0}_t})\leq \E\{\mathrm{Ent}(\L_{X^{\nu_0}_t|\F_t^B}|\L_{X^{\mu_0}_t|\F_t^B})\}
\end{align}
implies \eqref{logHa}, which combined with \cite[Theorem 1.4.2(2)]{Wbook} completes the proof.
 \end{proof}
\subsection{Conditional propagation of chaos}
Fix $T>0$. Let $(\Omega^i, \scr F^i, (\scr F^i_t)_{t\geq 0},\P^i), i=0,1$ and $(\Omega, \scr F, (\scr F_t)_{t\geq 0},\P)$ be defined in Section 1. Let $N\ge1$ be an integer, $(W^i_t)_{1\le i\le N}$ be $N$ independent $d_W$-dimensional Brownian motions on $(\Omega^1, \scr F^1, (\scr F^1_t)_{t\geq 0},\P^1)$, $B_t$ be a $d_B$-dimensional Brownian motion on $(\Omega^0, \scr F^0, (\scr F^0_t)_{t\geq 0},\P^0)$, and $(X_0^i)_{1\le i\le N}$ be i.i.d. $(\Omega^1,\F_0^1)$-measurable $\R^d$-valued random variables. Let $b:[0,T]\times \R^d\times\scr P(\R^d)\to\R^d$, $\sigma:[0,T]\times \R^d\times\scr P(\R^d)\to\R^d\otimes\R^{d_W}$ and $\tilde{\sigma}:[0,T]\to\R^d\otimes\R^{d_B}$ be measurable. Consider conditional McKean-Vlasov SDEs
\begin{align}\label{GPS}\d X_t^i= b_t(X_t^i, \L_{X_t^i|\F_{t}^B})\d t+  \sigma_t(X^i_t,\L_{X_t^i|\F_{t}^B}) \d W^i_t+ \tt \si_t\d  B_t,\ \ 1\leq i\leq N,
\end{align}
and the mean field interacting particle system with common noise
\begin{align}\label{GPS00}\d X^{i,N}_t=b_t(X_t^{i,N}, \hat\mu_t^N)\d t+\sigma_t(X^{i,N}_t, \hat\mu_t^N) \d W^i_t+ \tt \si_t\d  B_t,\ \ 1\leq i\leq N,
\end{align}
where for any $1\leq i\leq N$, $X_0^{i,N}$ is $(\Omega^1,\F_0^1)$-measurable $\R^d$-valued random variable, the distribution of $(X_0^{1,N},X_0^{2,N},\cdots,X_0^{N,N})$ is exchangeable and $\hat\mu_t^N$ is the empirical distribution of $(X_t^{i,N})_{1\leq i\leq N}$, i.e.
\begin{equation*}
 \hat\mu_t^N =\ff{1}{N}\sum_{j=1}^N\dd_{X_t^{j,N}}.
 \end{equation*}
\begin{thm}\label{POC}
Assume $b,\sigma,\tilde{\sigma}$ are locally bounded and
\begin{align}\label{HTL00}
\nonumber&|b_t(x,\mu)-b_t(y,\nu)| +\|\sigma_t(x,\mu)-\sigma_t(y,\nu)\|_{HS}\\
&\leq K(|x-y|+\W_2(\mu,\nu)),\ \ t\in[0,T], x,y\in\R^d, \mu,\nu\in\scr P_2(\R^d).
\end{align}
Then the following assertions hold.

(1) Assume that $\E|X_0^1|^{q}<\infty$ for some $q>2$.
Then there exists a constant $C>0$ depending only on $d,T$ and $\E|X_0^1|^{q}$ such that
\begin{equation}\begin{split}\label{S1}
& \sup_{t\in[0,T]}\E\W_2(\hat{\mu}_t^N,\L_{X_t^1|\F_t^B}^{\P})^2\\
&\le C\frac{1}{N}\W_2(\L_{(X_{0}^{1},X_{0}^{2},\cdots, X_{0}^{N})},\L_{(X_{0}^{1,N},X_{0}^{2,N},\cdots, X_{0}^{N,N})})^2+C R_{d,q}(N),
\end{split}\end{equation}
and
\begin{align}
\label{Wao}\nonumber&\E\W_2(\L^{\P}_{(X_{t}^{1},X_{t}^{2},\cdots, X_{t}^{k})|\F_t^B},\L^{\P}_{(X_{t}^{1,N},X_{t}^{2,N},\cdots, X_{t}^{k,N})|\F_t^B})^2\\
&\leq C\frac{k}{N}\W_2(\L_{(X_{0}^{1},X_{0}^{2},\cdots, X_{0}^{N})},\L_{(X_{0}^{1,N},X_{0}^{2,N},\cdots, X_{0}^{N,N})})^2+CkR_{d,q}(N),
\end{align}
where
\begin{equation*}\begin{split}
R_{d,q}(N)=
\begin{cases}
N^{-\ff{1}{2}}+N^{-\frac{q-2}{q}},~~~~~~~~~~~~~~~~~~~~d<4,q\neq 4,\\
N^{-\ff{1}{2}}\log (1+N)+N^{-\frac{q-2}{q}},~~~~ ~~~d=4, q\neq 4,\\
N^{-\ff{2}{d}}+N^{-\frac{q-2}{q}},~~~~~~~~~~~~~~~~~~~~d>4, q\neq\frac{d}{d-2}.
\end{cases}
\end{split}\end{equation*}
(2) If in addition, $\sigma_t(x,\mu)$ does not depend on $\mu$ and $\lambda^{-1}\geq\sigma\sigma^\ast\geq \lambda$ for some $\lambda\in(0,1]$, then for any $k\geq 1$ and $k\leq N$, it holds
\begin{align}\label{VEN}
&\E\mathrm{Ent}(\L^{\P}_{(X_{t}^{1,N},X_{t}^{2,N},\cdots, X_{t}^{k,N})|\F_t^B}|\L^{\P}_{(X_{t}^{1},X_{t}^{2},\cdots,X_{t}^{k})|\F_t^B})\\
\nonumber&\leq CkR_{d,q}(N)+\frac{C}{1-\e^{-(K^2+2K)t}}\frac{k}{N}\W_2(\L_{(X_{0}^{1},X_{0}^{2},\cdots, X_{0}^{N})},\L_{(X_{0}^{1,N},X_{0}^{2,N},\cdots, X_{0}^{N,N})})^2.
\end{align}
\end{thm}
\begin{proof}
(1) 
Let $\P^{B},\E^{B}$ be in \eqref{PTY}. It is standard to derive from \eqref{HTL00} that
\begin{align}\label{mmo} \E^{B}[\sup_{t\in[0,T]}|X^{1}_t|^q]<c_1\left(1+\E|X^1_0|^q+\left|\int_0^T\tilde{\sigma}_t\d B_t\right|^q\right)
\end{align}
for some constant $c_1>0$ depending on $q, T$.
 Denote $\mu_t^i=\L_{X_t^i|\F_t^B}^{\P},i\geq 1$. Since \eqref{GPS} is well-posed under \eqref{HTL00} due to Lemma \ref{EEB}, $\mu_t^i$ does not depend on $i$ and we write $\mu_t=\mu_t^i,1\leq i\leq N$.
 Letting $\tt\mu_t^N =\ff{1}{N}\sum_{j=1}^N\dd_{X_t^j}$, we obtain from the triangle inequality that
\begin{equation}\label{A4}
\begin{split}
\W_2(\hat\mu^N _s,\mu_s)&\le
\mathbb{W}_2(\hat\mu^N_s,\tt\mu^N_s)+\mathbb{W}_2(\tt\mu^N_s,\mu_s)\\
&\le
\left(\frac{1}{N}\sum_{i=1}^N|X^{i,N}_s-X^i_s|^2\right)^{\frac{1}{2}} +\mathbb{W}_2(\tt\mu^N_s,\mu_s).
\end{split}
\end{equation}
By BDG's inequality, \eqref{HTL00} and \eqref{A4}, there exists a constant $c_2>0$ such that
\begin{equation*}
\begin{split}
\sum_{i=1}^N\E^{B}\sup_{s\in[0,t]}|X^{i,N}_s-X^{i}_s|^2
&\le c_2\sum_{i=1}^N\E|X^{i,N}_0-X^{i}_0|^2+c_2\int_0^t\E^{B}\sum_{i=1}^N|X^{i,N}_s-X^{i}_s|^2\d s\\
&+c_2N\E^{B}\int_0^t\W_2(\tilde{\mu}^N_s,\mu_s)^2\d s.
\end{split}
\end{equation*}
By Gronwall's inequality, we can find a constant $c_3>0$ such that
\begin{equation}\label{LEW00}
\sum_{i=1}^N\E^{B}\sup_{s\in[0,t]}|X^{i,N}_s-X^{i}_s|^2\le c_3\sum_{i=1}^N\E|X^{i,N}_0-X^{i}_0|^2+c_3N\E^{B}\int_0^t\mathbb{W}_2(\tt\mu^N_s,\mu_s)^2\d s.
\end{equation}
By \cite[Theorem 1]{FG} for $p=2$ and \eqref{mmo}, there exists a constant $C_0>0$ depending only on $q,d$ such that
\begin{align}\label{MYt}
\nonumber&\E^{B}\mathbb{W}_2(\tt\mu^N_s,\mu_s)^2\le C_0 \left(\E^{B}[\sup_{t\in[0,T]}|X^{1}_t|^q]\right)^{\frac{2}{q}}
R_{d,q}(N)\\
&\leq C_0 c_1^{\frac{2}{q}}\left(1+\E|X^1_0|^q+\left|\int_0^T\tilde{\sigma}_t\d B_t\right|^q\right)^{\frac{2}{q}}R_{d,q}(N),\ \ s\in[0,T].
\end{align}
So, we derive \eqref{Wao} by combining with \eqref{LEW00} and the fact
\begin{align}
\label{PUk}\nonumber&\W_2(\L_{(X_{t}^{1},X_{t}^{2},\cdots, X_{t}^{k})|\F_t^B},\L_{(X_{t}^{1,N},X_{t}^{2,N},\cdots, X_{t}^{k,N})|\F_t^B})^2\\
&\leq \frac{k}{N}\W_2(\L_{(X_{t}^{1},X_{t}^{2},\cdots, X_{t}^{N})|\F_t^B},\L_{(X_{t}^{1,N},X_{t}^{2,N},\cdots, X_{t}^{N,N})|\F_t^B})^2.
\end{align}
Finally, \eqref{MYt} together with \eqref{LEW00} and \eqref{A4} derives
\begin{align*}
\nonumber\E^{B}\mathbb{W}_2(\hat{\mu}^N_s,\mu_s)^2&\le C_1\frac{1}{N}\sum_{i=1}^N\E|X^{i,N}_0-X^{i}_0|^2\\
&+C_1\left(1+\E|X^1_0|^q+\left|\int_0^T\tilde{\sigma}_t\d B_t\right|^q\right)^{\frac{2}{q}}R_{d,q}(N),\ \ s\in[0,T],
\end{align*}
for some constant $C_1>0$ depending on $d,T$, which yields \eqref{S1} by taking expectation with respect to $\P$.

(2) {\bf(Step (i))} We first assume that $b$ is bounded.
Define
\begin{align}\label{PT0}\P^{B,0}:= \P(\ \cdot\ |\F_T^B\bigvee\F_0),\ \ \E^{B,0}:= \E(\ \cdot\ | \F_T^B\bigvee\F_0).\end{align}
Consider
\begin{align}\label{GPSer}\d \bar{X}_t^{i}&= b_t(\bar{X}_t^{i}, \mu_t)\d t+  \sigma_t(\bar{X}^{i}_t) \d W^i_t+\tt \si_t\d  B_t,\ \ t\in [0,T], \bar{X}_0^{i}=X_0^{i,N}, 1\leq i\leq N.
\end{align}
Rewrite \eqref{GPSer} as
\begin{align}\label{MXY}\nonumber  \d \bar{X}_t^{i}&= b_t(\bar{X}_t^{i}, \frac{1}{N}\sum_{i=1}^N\delta_{\bar{X}_t^{i}})\d t+  \sigma_t(\bar{X}^{i}_t) \d \hat{W}^i_t+\tt \si_t\d  B_t,\ \ t\in [0,T], 1\leq i\leq N.
\end{align}
where for $1\leq i\leq N$,
\begin{equation*}\begin{split}
&\hat{W}_t^i :=   W_t^i-\int_0^t \gamma_s^i\d s,\\ &\gamma_t^i:=[\sigma_t^\ast(\sigma_t\sigma_t^\ast)^{-1}](\bar{X}^{i}_t)\left(b_t(\bar{X}_t^{i}, \frac{1}{N}\sum_{i=1}^N\delta_{\bar{X}_t^{i}})-b_t(\bar{X}_t^{i}, \mu_t)\right), \ \ t\in[0,T].
\end{split}\end{equation*}
Fix $t_0\in[0,T]$. Let
\begin{equation*}\begin{split}
&\gamma_t=(\gamma_t^1,\gamma_t^2,\cdots, \gamma_t^N),\ \ \hat{W}_t=(\hat{W}_t^1,\hat{W}_t^2,\cdots, \hat{W}_t^N),\\
& R_t:= \e^{\int_0^{t}\sum_{i=1}^N\<\gamma_r^i, \d W^i _r\> -\ff 1 2 \int_0^{t} \sum_{i=1}^N|\gamma_r^i|^2\d r},\\
& \Q_t^{B,0}:= R_t\P^{B,0}, \ \ t\in [0,t_0].
\end{split}\end{equation*}
Since $b$ is bounded, we can apply Girsanov's theorem to conclude that
$\{\hat{W}_{t}\}_{t\in[0,t_0]}$ is an $(N\times d_{W})$-dimensional Brownian motion under the weighted conditional probability $\Q^{B,0}_{t_0}$.
So, we have
\begin{align*}
&\L^{\Q^{B,0}_{t_0}}_{(\{\bar{X}^i_t\}_{1\leq i\leq N},B_t)_{t\in[0,t_0]}}=\L^{\P^{B,0}}_{(\{X_t^{i,N}\}_{1\leq i\leq N},B_t)_{t\in[0,t_0]}},\\
&\L^{\Q^{B,0}_{t_0}}_{(\{\bar{X}^i_t\}_{1\leq i\leq N},B_t)_{t\in[0,t_0]}|\F_T^B}=\L^{\P^{B,0}}_{(\{X_t^{i,N}\}_{1\leq i\leq N},B_t)_{t\in[0,t_0]}|\F_T^B}.
\end{align*}
This gives
\begin{align*}&\E_{\Q^{B,0}_{t_0}}\int_0^{t_0}\W_2(\mu_t,\frac{1}{N}\sum_{i=1}^N\delta_{\bar{X}_t^i})^2\d t=\E^{B,0}\int_0^{t_0}\W_2(\mu_t,\frac{1}{N}\sum_{i=1}^N\delta_{X_t^{i,N}})^2\d t,
\end{align*}
which together with \eqref{HTL00}, $\lambda\leq \sigma\sigma^\ast\leq \lambda^{-1}$ and Young's inequality implies that
\beg{align}\label{DDT} \nonumber&\E^{B,0} \log F(X_{t_0}^{1,N},X_{t_0}^{2,N},\cdots, X_{t_0}^{N,N})\\
\nonumber&\le \log \E^{B,0} [ F(\bar{X}_{t_0}^{1},\bar{X}_{t_0}^{2},\cdots,\bar{X}_{t_0}^{N})]+\E^{B,0}(R_{t_0}\log R_{t_0})\\
&\le \log \E^{B,0} [ F(\bar{X}_{t_0}^{1},\bar{X}_{t_0}^{2},\cdots,\bar{X}_{t_0}^{N})]+\frac{1}{2}\sum_{i=1}^N\E_{\Q^{B,0}_{t_0}}\int_0^{t_0}|\gamma^i_t|^2\d t\\
\nonumber&\leq \log \E^{B,0} [ F(\bar{X}_{t_0}^{1},\bar{X}_{t_0}^{2},\cdots,\bar{X}_{t_0}^{N})]\\
\nonumber &+ c_1N\E^{B,0}\int_0^{t_0}\W_2(\mu_t,\frac{1}{N}\sum_{i=1}^N\delta_{X_t^{i,N}})^2\d t, \ \ 0< F\in \B_b((\R^d)^N).\end{align}
for some constant $c_1>0$.
On the other hand,
let $$b_t^{\mu,B}(x)=b_t\left(x+\int_0^t\tilde{\sigma}_s\d B_s,\mu_t\right),\ \ \sigma_t^B(x)=\sigma_t\left(x+\int_0^t\tilde{\sigma}_s\d B_s\right), \ \ t\in[0,T], x\in\R^d.$$
Consider
\begin{align*}
\d Y_t^{i}=b_t^{\mu,B}(Y_t^{i})+\sigma_t(Y_t^{i})\d W_t^i,\ \ 1\leq i\leq N, Y_0^i=X_0^i,
\end{align*}
and
\begin{align*}
\d \bar{Y}_t^{i}=b_t^{\mu,B}(\bar{Y}_t^{i})+\sigma_t(\bar{Y}_t^{i})\d W_t^i,\ \ 1\leq i\leq N, \bar{Y}_0^i=X_0^{i,N}.
\end{align*}
By \cite[Theorem 3.4.1]{Wbook} and \eqref{HTL00}, for large enough $p>1$ , we get Wang's Harnack inequality with power $p$ for some constant $c(p)>0$:
\begin{align*}&\left(\E^{B,0} [\bar{F}(\bar{Y}_t^{1},\bar{Y}_t^{2},\cdots,\bar{Y}_t^{N})]\right)^p
\leq \E^{B,0} [\bar{F}(Y_t^{1},Y_t^{2},\cdots,Y_t^{N})^p]\\
&\qquad\quad\quad\qquad\quad\times \exp\left\{\frac{c(p)\sum_{i=1}^N|X_0^{i,N}-X_0^{i}|^2}{1-\e^{-(K^2+2K)t}}\right\}, \ \ \bar{F}\in \scr B_b^+((\R^d)^N),t\in(0,T].
\end{align*}
In view of $\bar{X}_t^i=\bar{Y}_t^i+\int_0^t\tilde{\sigma}_s\d B_s$ and $X_t^i=Y_t^i+\int_0^t\tilde{\sigma}_s\d B_s$,
we conclude that
\begin{align*}&\left(\E^{B,0} [F(\bar{X}_t^{1},\bar{X}_t^{2},\cdots,\bar{X}_t^{N})]\right)^p\leq \E^{B,0} [F(X_t^{1},X_t^{2},\cdots,X_t^{N})^p]\\
&\qquad\quad\quad\qquad\quad\times \exp\left\{\frac{c(p)\sum_{i=1}^N|X_0^{i,N}-X_0^{i}|^2}{1-\e^{-(K^2+2K)t}}\right\}, \ \ F\in \scr B^+_b((\R^d)^N),t\in(0,T].
\end{align*}
This together with \eqref{DDT}, \cite[Theorem 1.4.2(1)-(2)]{Wbook} and \cite[Lemma 2.1]{23RW} implies that
\begin{align}\label{DDG}
\nonumber&\E^{B,0} \log F(X_{t_0}^{1,N},X_{t_0}^{2,N},\cdots, X_{t_0}^{N,N})\\
&\leq  \log \E^{B,0} [ F(X_{t_0}^{1},X_{t_0}^{2},\cdots,X_{t_0}^{N})]+ c_1pN\E^{B,0}\int_0^{t_0}\W_2(\mu_t,\frac{1}{N}\sum_{i=1}^N\delta_{X_t^{i,N}})^2\d t\\
\nonumber&+\frac{c(p)\sum_{i=1}^N|X_0^{i,N}-X_0^{i}|^2}{1-\e^{-(K^2+2K)t_0}}, \ \ 0< F\in \B_b((\R^d)^N).
\end{align}
Taking expectation with respect to $\E^B$ on both sides and using Jensen's inequality, we derive
\begin{align}\label{ENP}
\nonumber&\E^{B} \log F(X_{t_0}^{1,N},X_{t_0}^{2,N},\cdots, X_{t_0}^{N,N})\\
&\leq  \log \E^{B} [ F(X_{t_0}^{1},X_{t_0}^{2},\cdots,X_{t_0}^{N})]+ c_1pN\E^{B}\int_0^{t_0}\W_2(\mu_t,\frac{1}{N}\sum_{i=1}^N\delta_{X_t^{i,N}})^2\d t\\
\nonumber&+\frac{c(p)\sum_{i=1}^N\E|X_0^{i,N}-X_0^{i}|^2}{1-\e^{-(K^2+2K)t_0}}, \ \ 0< F\in \B_b((\R^d)^N).
\end{align}
For any $1\leq k\leq N$ and $0< f\in \B_b((\R^d)^k)$, take $$F_f(x_1,x_2,\cdots,x_{\lfloor\frac{N}{k}\rfloor k})=\prod_{i=0}^{\lfloor \frac{N}{k}\rfloor-1}f(x_{ik+1},x_{ik+2},\cdots,x_{ik+k}),$$
where $\lfloor \frac{N}{k}\rfloor$ stands for the integer part of $\frac{N}{k}$.
Since $(X_{t_0}^{1,N},X_{t_0}^{2,N},\cdots,X_{t_0}^{N,N})$ is exchangeable and $X_{t_0}^{1},X_{t_0}^{2},\cdots,X_{t_0}^{N}$ are i.i.d. under $\P^B$ and $\lfloor \frac{N}{k}\rfloor^{-1}\leq \frac{2k}{N},1\leq k\leq N$, we derive from \eqref{ENP} for $F=F_f$ that
\beg{align}\label{DDF}\nonumber &\E^{B} \log f(X_{t_0}^{1,N},X_{t_0}^{2,N},\cdots, X_{t_0}^{k,N})\\
&\le \log \E^{B} [ f(X_{t_0}^{1},X_{t_0}^{2},\cdots,X_{t_0}^{k})]+\frac{2c(p)k}{1-\e^{-(K^2+2K)t_0}}\frac{1}{N}\sum_{i=1}^N\E|X_0^{i,N}-X_0^{i}|^2\\
\nonumber &+ 2c_1pk\E^{B}\int_0^{t_0}\W_2(\mu_t,\frac{1}{N}\sum_{i=1}^N\delta_{X_t^{i,N}})^2\d t, \ \ 0< f\in \B_b((\R^d)^k).\end{align}
Again using \cite[Theorem 1.4.2(2)]{Wbook}, we derive \eqref{VEN} from \eqref{S1} and \eqref{DDF}.

{\bf(Step (ii))} In general, let $b^{(n)}=(-n\vee b^i\wedge n)_{1\leq i\leq d}, n\geq 1$. Noting that \eqref{HTL00} holds for $b^{(n)}$ in place of $b$, \eqref{VEN} follows from {\bf Step (i)} and an approximation technique.
\end{proof}
\begin{rem} (1) Note that in the present case, the coefficients are only assumed to be Lipschitz continuous in $\W_2$-distance with respect to the measure variable so that \cite[Theorem 1]{FG} for $p=2$ is used to estimate the convergence rate of conditional propagation of chaos, which depends on the dimension $d$ and seems a little complicated. One can also refer to \cite[Theorem 2.12]{CD} for the case $q>4$ and $X_0^{i,N}=X_0^i,1\leq i\leq N$. However, if we only consider the special case:
 $$b_t(x,\mu)=\int_{\R^d}\tilde{b}_t(x,y)\mu(\d y),\ \ \sigma_t(x,\mu)=\int_{\R^d}\tilde{\sigma}_t(x,y)\mu(\d y)$$
 for some Lipschitz continuous functions $\tilde{b},\tilde{\sigma}$ uniformly in time variable $t$,
 the convergence rate in Theorem \ref{POC} and Theorem \ref{SHS} below can be improved to be $R_{d,q}(N)=\frac{1}{N}$ and $q=2$.

 (2) In Theorem \ref{POC}(2), the coefficients before the private noise can depend on the spatial variable and the initial distribution of interacting particle system \eqref{GPS00} is allowed to be singular with that of the conditional McKean-Vlasov SDEs \eqref{GPS} since \eqref{VEN} only involves in $\frac{1}{N}\W_2(\L_{(X_{0}^{1},X_{0}^{2},\cdots, X_{0}^{N})},\L_{(X_{0}^{1,N},X_{0}^{2,N},\cdots, X_{0}^{N,N})})^2$. See also \cite{BJW,JW,JW1} for the quantitative propagation of chaos in relative entropy by the entropy method in the additive noise case and under the assumption $$\lim_{N\to\infty}\frac{\mathrm{Ent}(\L_{(X_0^{1,N},X_0^{2,N},\cdots,X_0^{N,N})}|\L_{(X_0^1,X_0^2,\cdots, X_0^N)})}{N}=0.$$
\end{rem}
\section{Conditional distribution dependent stochastic Hamiltonian system}
In this part, we consider conditional distribution dependent stochastic Hamiltonian system with additive noise, which is a type of degenerate model. More precisely, we consider
 \beq\label{E00}
\begin{cases}
\d X_t^{(1)}=\big\{AX^{(1)}_t+MX_t^{(2)}\big\}\d t, \\
\d X_t^{(2)}=b_t(X_t,\L_{X_t|\F_t^{B}})\d t+\sigma_t\d W_t+ \tt\sigma_t(\L_{X_t|\F_t^{B}})\d B_t,\ \ t\in [0,T],
\end{cases}
\end{equation} where $ X_t=(X_t^{(1)}, X_t^{(2)})$, $W_t,B_t$ are given in Section 1,
$b: [0,T]\times \R^{m+d}\times \scr P(\R^{m+d})\to \R^d,\ \ \si: [0,T]\to \R^{d}\otimes \R^{d_W}, \ \ \tt \si: [0,T]\times \scr P(\R^{m+d})\to \R^d\otimes \R^{d_B}$ are measurable and bounded on bounded sets, $A$ is an $m\times m$ matrix and $M$ is an $m\times d$ matrix.
\subsection{Log-Haranck inequality}
To establish the log-Harnack inequality, we make the following assumption.

 \begin{enumerate}
\item[\bf{(C)}] $(\sigma_t\sigma_t^\ast)^{-1}$ is bounded in $t\in[0,T]$, and there exists constants $K,\tilde{K}>0$ such that
\begin{align*}&|b_t(x,\mu)-b_t(y,\nu)| \le  K (|x-y|+\W_2(\mu, \nu)),\\
&\|\tt\si_t (\mu)-\tt\si_t (\nu)\|\leq \tilde{K} \W_2(\mu, \nu),\ \ t\in [0,T],\ x,y\in\R^d,\ \mu,\nu\in \scr P_2(\R^{m+d}).
\end{align*}
    Moreover, the following Kalman's rank condition holds for some integer $1\leq l\leq m$:
\begin{align*}\mathrm{Rank}[A^iM, 0\le i\le l-1]=m.\end{align*} \end{enumerate}

By Lemma \ref{EEB}, {\bf (C)} implies that   \eqref{E00} is   well-posed.
As in \cite{HW22+},
for any $t>0$, we consider the modified distance
$$\rr_t(x,y):= \ss{t^{-2}|x^{(1)}-y^{(1)}|^2+|x^{(2)}-y^{(2)}|^2},\ \ \ x=(x^{(1)},x^{(2)}),y=(y^{(1)},y^{(2)})\in\R^{m+d},$$
and define the associated $L^2$-Wasserstein distance
$$\W_{2,t}(\mu,\nu):=\inf_{\pi\in \C(\mu,\nu)}\bigg( \int_{\R^{m+d}\times \R^{m+d}} \rr_t(x,y)^2 \pi(\d x,\d y)\bigg)^{\ff 1 2}. $$
The next theorem characterizes the log-Harnack inequality and the proof is similar to that of \cite[Theorem 3.1]{HW22+} since $\int_0^s\tt\sigma_t(\L_{X_t|\F_t^{B}})\d B_t$ is deterministic given $B$. Hence, we will give an outline of the procedure in the following.
\begin{thm}\label{LHI} Assume {\bf (C)}.
Then there exists a constant $ c>0$ such that for any $0<f\in \B_b(\R^{m+d}), \mu_0,\nu_0 \in \scr P_2(\R^{m+d}), t\in (0,T]$ and $\xi,\tilde{\xi}\in L^2(\Omega^1\to\R^{m+d},\F^1_0,\P^1)$ with $\L_{\xi}=\mu_0, \L_{\tilde{\xi}}=\nu_0$,
\beq\label{ECt}  \beg{split}&\E\{\mathrm{Ent}(\L_{X^{\xi}_t|\F_t^B}|\L_{X^{\tilde{\xi}}_t|\F_t^B})\}
 \le \ff{c}{t^{4l-3}}\W_{2,t}(\mu_0,\nu_0)^2  \le \ff{c(1\lor T^2)}{t^{4l-1}}\W_{2}(\mu_0,\nu_0)^2,\end{split}\end{equation}
and consequently,
\begin{align}\label{EC2}  &P_t\log f(\nu_0) -   \log P_tf(\mu_0)
 \le \ff{c}{t^{4l-3}}\W_{2,t}(\mu_0,\nu_0)^2  \le \ff{c(1\lor T^2)}{t^{4l-1}}\W_{2}(\mu_0,\nu_0)^2.\end{align}
 \end{thm}

 \begin{proof}
For any $t_0\in (0,T]$ and $\mu_0,\nu_0\in \scr P_2(\R^{m+d})$, let $X_0^{\mu_0},X_0^{\nu_0}$ be $(\Omega^1,\F_0^1)$-measurable such that
 \beq\label{0'} \L_{X_0^{\mu_0}}=\mu_0,\ \ \L_{X_0^{\nu_0}}=\nu_0,\ \ \E[\rr_{t_0}(X_0^{\mu_0},X_0^{\nu_0})^2]=\W_{2,t_0}(\mu_0,\nu_0)^2.\end{equation}
Let $X_t^{\mu_0}$ and $X_t^{\nu_0}$ solve \eqref{E00} with initial value $X_0^{\mu_0}$ and $X_0^{\nu_0}$ respectively and let $\mu_t,\nu_t, \eta^\mu$ be in \eqref{mnu} and \eqref{XM}.
Then \eqref{-2} and \eqref{etn} still hold.
 For fixed $t_0\in (0,T],$ let
\beq\label{A1}\beg{split} &Q_{t} :=\int_0^{t} \ff{s(t-s)}{t^2}\e^{-sA}MM^*\e^{-sA^*}\d s,\ \ t\in[0,t_0]\\
&v=X_0^{\nu_0}-X_0^{\mu_0},\\
&V_{t_0}^{\mu,\nu}:= \int_0^{t_0}\e^{-rA}M\Big\{ \ff{t_0-r}{t_0} v^{(2)} +\ff r {t_0}\big(\eta_{t_0}^\mu-\eta_{t_0}^\nu\big) +\eta_r^\nu-
\eta_r^\mu\Big\}\d r,\\
  &\aa_{t_0}(s) := \ff{s}{t_0} \big(\eta_{t_0}^\mu-\eta_{t_0}^\nu -v^{(2)}\big)\\
  &\quad\quad\qquad- \ff{s(t_0-s)}{t_0^2} M^*\e^{-sA^*}Q_{t_0}^{-1}\big(v^{(1)}+ V_{t_0}^{\mu,\nu}\big),\ \ s\in[0,t_0].\end{split} \end{equation}
Denote $Y_t=(Y_t^{(1)},Y_t^{(2)})$ the solution to the SDE:
\beq\label{A2} \begin{cases}
\d Y_t^{(1)}=\big\{AY_t^{(1)}  +MY_t^{(2)}\big\}\d t, \\
\d Y_t^{(2)}=\big\{b_t(X^{\mu_0}_t,\mu_t)+  \aa_{t_0}'(t) \big\}\d t+\sigma_t\d W_t+\tt\si_t(\nu_t)\d B_t,\ \ Y_0=X_0^{\nu_0}.
    \end{cases}\end{equation}
which combined with \eqref{A1} yields $Y_{t_0}=X^{\mu_0}_{t_0}$.
Let
\begin{equation*}\gamma_s := \sigma_s^{\ast}(\sigma_s\sigma^\ast_s)^{-1}\big\{b_s(Y_s, \nu_s)- b_s(X_s^{\mu_0}, \mu_s)  -  \aa_{t_0}'(s)\big\},\ \ s\in [0,t_0].\end{equation*}
By {\bf (C)} and \cite[(3.17)]{HW22+}, there exists a constant $c_1>0$ uniformly in $t_0\in (0,T]$  such that
\beq\label{AP5}  \beg{split}  |\gamma_s |^2\le &\,c_1
\Big\{\W_2(\mu_s,\nu_s)^2  + t_0^{4(1-l)}   \rr_{t_0}(X_0^{\mu_0},X_0^{\nu_0})^2   + t_0^{4(1-l)} \sup_{t\in [0,t_0]} |\eta_t^\mu-\eta_t^\nu|^2\Big\}\\
&  +  c_1 t_0^{2-4l}\Big(\rr_{t_0}(X_0^{\mu_0},X_0^{\nu_0})^2+\sup_{t\in [0,t_0]} |\eta_t^\nu-\eta_t^\mu|^2\Big), \ \ s\in[0,t_0].\end{split}\end{equation}
Recall that $\P^{B,0}$ is defined in \eqref{PT0}. By Girsanov's theorem,
$$\hat{W}_t:=W_t- \int_0^t \gamma_s \d s,\ \ t\in [0,t_0]$$
is a $d_W$-dimensional Brownian motion under the weighted  conditional probability measure $\d \Q^{B,0}:=R \d\P^{B,0}$, where
$$R:= \e^{\int_0^{t_0} \<\gamma_s,\d W_s\>-\ff 1 2 \int_0^{t_0} |\gamma_s|^2\d s}.$$
 By  \eqref{A2}, $\hat Y_t:= Y_t-(0,\eta_t^\nu)$ solves the SDE
$$\begin{cases}
\d \hat Y_t^{(1)}=\big\{A\hat Y_t^{(1)}  +M\hat Y_t^{(2)}+M\eta_t^\nu\big\}\d t, \\
\d \hat Y_t^{(2)}= b_t(\hat Y_t+(0,\eta_t^\nu),\nu_t) \d t+\sigma_t\d \hat{W}_t,\ \ t\in [0,t_0],\ \hat Y_0=Y_0.
\end{cases}$$
Observe that $\hat X_t^\nu:= X_t^\nu-(0,\eta_t^\nu)$  solves the same  equation  as $\hat{Y}_t$  for $W_t$ replacing $\hat{W}_t$.
By the weak uniqueness and  the fact that $\eta_t^\nu$ is $\F_T^B$-measurable, we get
$$\L^{\Q^{B,0}}_{Y_{t_0}}= \L^{\Q^{B,0}}_{\hat Y_{t_0} +(0,\eta_{t_0}^\nu)} = \L^{\P^{B,0}}_{\hat X_{t_0}^{\nu_0} +(0,\eta_{t_0}^\nu)} = \L^{\P^{B,0}}_{X_{t_0}^{\nu_0}}.$$
This together with $Y_{t_0}=X^{\mu_0}_{t_0}$, Young's inequality and \eqref{AP5} yields that  we find  some constant $c_2>0$ such that for any $0<f\in \B_b(\R^{m+d})$,
\begin{equation*} \beg{split} &\E^{B,0} [\log f(X_{t_0}^{\nu_0})]=\E^{B,0} [R \log f(Y_{t_0})]=\E^{B,0} [R \log f(X_{t_0}^{\mu_0})]\\
&\leq \log \E^{B,0} [f(X_{t_0}^{\mu_0})] + \E^{B,0} [R\log R]\\
 &= \log \E^{B,0} [f(X_{t_0}^{\mu_0})]+\ff 1 2\E_{\Q^{B,0}} \int_0^{t_0}|\gamma_t|^2\d t\\
&\le\log \E^{B,0} [f(X_{t_0}^{\mu_0})]\\
&\quad+ c_2 \Big\{t_0\sup_{s\in[0,t_0]}\W_2(\mu_s,\nu_s)^2  +  t_0^{3-4l}   \rr_{t_0}(X_0^{\mu_0},X_0^{\nu_0})^2+ t_0^{3-4l}\sup_{t\in [0,t_0]} |\eta_t^\nu-\eta_t^\mu|^2\Big\}.  \end{split} \end{equation*}
By taking expectation with respect to $\P^B$, using Jensen's inequality,
we obtain
\begin{equation*} \beg{split} &\E^{B} [\log f(X_{t_0}^{\nu_0})]\le\log \E^{B} [f(X_{t_0}^{\mu_0})]\\
&\quad+ c_2 \Big\{t_0\sup_{s\in[0,t_0]}\W_2(\mu_s,\nu_s)^2  +  t_0^{3-4l}   \E \rr_{t_0}(X_0^{\mu_0},X_0^{\nu_0})^2+ t_0^{3-4l}\sup_{t\in [0,t_0]} |\eta_t^\nu-\eta_t^\mu|^2\Big\}.  \end{split} \end{equation*}
Then by Lemma \ref{EEB}, \eqref{-2}, \eqref{etn}, \cite[(3.5)]{HW22+}, \cite[Theorem 1.4.2(2)]{Wbook} and \eqref{0'},  we prove \eqref{ECt}, which gives \eqref{EC2} by \eqref{eng} and \cite[Theorem 1.4.2(2)]{Wbook}.
  \end{proof}
\subsection{Conditional propagation of chaos}
Let $N$ be a positive integer and $(X_0^i,W^i_t)_{1\le i\le N}$, $(X_0^{i,N})_{1\leq i\leq N}$ and $B_t$ be defined in the same way as in Section 2.2 with $m+d$ replacing $d$.
Let $X_t^i=(X_t^{i,(1)},X_t^{i,(2)})$ solve the conditional distribution dependent stochastic Hamiltonian system:
 \beq\label{E01}
\begin{cases}
\d X_t^{i,(1)}=\big\{AX^{i,(1)}_t+MX_t^{i,(2)}\big\}\d t, \\
\d X_t^{i,(2)}=b_t(X_t^i,\L_{X_t^i|\F_t^{B}})\d t+\sigma_t\d W_t^i+ \tt\sigma_t\d B_t,\ \ t\in [0,T],
\end{cases}
\end{equation}
and consider the mean field interacting particle system with common noise:
 \beq\label{E02}
\begin{cases}
\d X_t^{i,N,(1)}=\big\{AX_t^{i,N,(1)}+MX_t^{i,N,(2)}\big\}\d t, \\
\d X_t^{i,N,(2)}=b_t(X_t^{i,N},\frac{1}{N}\sum_{i=1}^N\delta_{X_t^{i,N}})\d t+\sigma_t\d W_t^i+ \tt\sigma_t\d B_t,\ \ t\in [0,T],
\end{cases}
\end{equation}
here $X_t^{i,N}=(X_t^{i,N,(1)}, X_t^{i,N,(2)})$.  Recall that $R_{d,q}(N)$ is defined in Theorem \ref{POC}.
\begin{thm}\label{SHS} Assume {\bf(C)} and $\E|X_0^1|^{q}<\infty$ for some $q>2$.
Then there exists a constant $C>0$ depending only on $d,T$ and $\E|X_0^1|^{q}$ such that for any $k\geq 1$ and $k\leq N$, it holds
\begin{align}\label{VVN}
&\E\mathrm{Ent}(\L^{\P}_{(X_{t}^{1,N},X_{t}^{2,N},\cdots, X_{t}^{k,N})|\F_t^B}|\L^{\P}_{(X_{t}^{1},X_{t}^{2},\cdots,X_{t}^{k})|\F_t^B})\\
\nonumber&\leq CkR_{d,q}(N)+\frac{C}{t^{4l-1}}\frac{k}{N}\W_2(\L_{(X_{0}^{1},X_{0}^{2},\cdots, X_{0}^{N})},\L_{(X_{0}^{1,N},X_{0}^{2,N},\cdots, X_{0}^{N,N})})^2,\ \ t\in(0,T].
\end{align}
\end{thm}
\begin{proof} Since {\bf(C)} implies \eqref{HTL00}, \eqref{S1} holds for \eqref{E01}-\eqref{E02} in place of \eqref{GPS}-\eqref{GPS00}. We first assume that $b$ is bounded. Fix $t_0\in(0,T]$. Let $Q_t$ be defined in \eqref{A1} and for $s\in[0,t_0]$,
\begin{equation*}\beg{split}
  &\aa_{t_0}^{i}(s) := -\ff{s}{t_0} (X_0^{i,N,(2)}-X_0^{i,(2)})\\
&\quad- \ff{s(t_0-s)}{t_0^2} M^*\e^{-sA^*}Q_{t_0}^{-1}\bigg((X_0^{i,N,(1)}-X_0^{i,(1)})\\
&\qquad\qquad\qquad\qquad\qquad\qquad\quad+ \int_0^{t_0}\e^{-rA}M\ff{t_0-r}{t_0} (X_0^{i,N,(2)}-X_0^{i,(2)})^{} \d r\bigg).\end{split} \end{equation*}
Again set $\mu_t=\mu_t^i=\L_{X_t^i|\F_t^{B}}$. Construct $Y_t^{i,N}=(Y_t^{i,N,(1)},Y_t^{i,N,(2)})$ as
 \beq\label{GPY}
\begin{cases}
\d Y_t^{i,N,(1)}=\big\{AY_t^{i,N,(1)}+MY_t^{i,N,(2)}\big\}\d t, \\
\d Y_t^{i,N,(2)}=[b_t(X_t^{i},\mu_t)+(\aa^i_{t_0})'(t)]\d t+\sigma_t\d W_t^i+ \tt\sigma_t\d B_t,\ \ Y_0^{i,N}=X_0^{i,N}.
\end{cases}
\end{equation}
Then we have
\begin{align*}&Y_t^{i,N,(2)}=X_t^{i,(2)}+(X_0^{i,N,(2)}-X_0^{i,(2)})+\alpha_{t_0}(t),\\
&Y_t^{i,N,(1)}=X_t^{i,(1)}+\e^{At}(X_0^{i,N,(1)}-X_0^{i,(1)})\\
&\qquad\quad+\int_0^t\e^{A(t-s)}M[(X_0^{i,N,(2)}-X_0^{i,(2)}) +\alpha_{t_0}(s)]\d s.
\end{align*}
In particular, it holds
\begin{align}\label{DDX}
Y_{t_0}^{i,N}=X_{t_0}^i,\ \ 1\leq i\leq N.
\end{align}
Let
\begin{align*}
&\d \hat{W}^i_t=\d W^i_t-\gamma_t^{i,N}\d t,\\ &\gamma_t^{i,N}=\sigma^\ast_t(\sigma_t\sigma_t^\ast)^{-1}\left[b_t(Y_t^{i,N}, \frac{1}{N}\sum_{i=1}^N\delta_{Y_t^{i,N}})-b_t(X_t^i, \mu_t)-\aa'_{t_0}(t)\right],\ \ 1\leq i\leq N.
\end{align*}
and
$$R_t=\exp\left\{\int_0^t\sum_{i-1}^N\<\gamma_s^{i,N},\d W^i_s\>-\frac{1}{2}\int_0^t\sum_{i=1}^N|\gamma_s^{i,N}|^2\d s\right\}.$$
Similarly to \eqref{AP5} and using the boundedness of $b$, there exists a constant $c_0>0$ such that
\begin{align}\label{KKS}\int_0^{t_0}|\gamma_s^{i,N}|^2\d s\leq c_0\int_0^{t_0}\left(1\wedge\W_2(\frac{1}{N}\sum_{i=1}^N\delta_{Y_t^{i,N}},\mu_t)^2\right)\d t+c_0\frac{|X_0^i-X_0^{i,N}|^2}{t_0^{4l-1}}.
\end{align}
By Girsanov's theorem, $((\hat{W}^i_t)_{t\in[0,t_0]})_{1\leq i\leq N}$ is an $(N\times d_W)$-dimensional Brownian motion under the conditional probability measure $\Q^{B,0}=R_{t_0}\P^{B,0}$.
Moreover, \eqref{GPY} can be rewritten as
\begin{equation*}
\begin{cases}
\d Y_t^{i,N,(1)}=\big\{AY_t^{i,N,(1)}+MY_t^{i,N,(2)}\big\}\d t, \\
\d Y_t^{i,N,(2)}=b_t(Y_t^{i,N},\frac{1}{N}\sum_{i=1}^N\delta_{Y_t^{i,N}})\d t+\sigma_t\d \hat{W}_t^i+ \tt\sigma_t\d B_t,\ \ Y_0^{i,N}=X_0^{i,N},t\in [0,t_0].
\end{cases}
\end{equation*}
By the weak uniqueness, it holds
$$\L^{\Q^{B,0}}_{\{Y_t^{i,N}\}_{1\leq i\leq N}}=\L^{\P^{B,0}}_{\{X_t^{i,N}\}_{1\leq i\leq N}},\ \ t\in[0,t_0].$$
This together with \eqref{DDX} implies
\begin{align*}\E^{B,0} f(X_{t_0}^{1,N},X_{t_0}^{2,N},\cdots, X_{t_0}^{N,N})&=\E^{B,0}[R_{t_0}f(Y_{t_0}^{1,N},Y_{t_0}^{2,N},\cdots, Y_{t_0}^{N,N})]\\
&=\E^{B,0}[R_{t_0}(X_{t_0}^{1},X_{t_0}^{2},\cdots, X_{t_0}^{N})],\ \ f\in \B_b((\R^{m+d})^N).
\end{align*}
Note that
$$\E_{\Q^{B,0}}\W_2(\frac{1}{N}\sum_{i=1}^N\delta_{Y_t^{i,N}},\mu_t^i)^2 =\E^{B,0}\W_2(\frac{1}{N}\sum_{i=1}^N\delta_{X_t^{i,N}},\mu_t^i)^2.$$
Combining this with Young's inequality and \eqref{KKS}, we can find a constant $c>0$ such that
\beg{align*} &\E^{B,0} \log f(X_{t_0}^{1,N},X_{t_0}^{2,N},\cdots, X_{t_0}^{N,N})\\
&\le \log \E^{B,0} [ f(X_{t_0}^{1},X_{t_0}^{2},\cdots,X_{t_0}^{N})]+\frac{1}{2}\E_{\Q^{B,0}}\int_0^{t_0}\sum_{i=1}^N|\gamma_s^{i,N}|^2\d s\\
&\le \log \E^{B,0} [ f(X_{t_0}^{1},X_{t_0}^{2},\cdots,X_{t_0}^{N})]\\
&+c\sum_{i=1}^N\E^{B,0}\int_0^{t_0}\W_2(\frac{1}{N}\sum_{i=1}^N\delta_{X_t^{i,N}},\mu_t)^2\d t+c\sum_{i=1}^N\frac{|X_0^i-X_0^{i,N}|^2}{t_0^{4l-1}}, \ \ 0< f\in \B_b((\R^{m+d})^N).\end{align*}
Repeating the same argument to derive \eqref{DDF} from \eqref{DDG}, we obtain
\beg{align*} &\E^B \log f(X_{t_0}^{1,N},X_{t_0}^{2,N},\cdots, X_{t_0}^{k,N})\\
&\le \log \E^B [ f(X_{t_0}^{1},X_{t_0}^{2},\cdots,X_{t_0}^{k})]\\
 &+ 2ck\left\{\E^B\int_0^{t_0}\W_2(\frac{1}{N}\sum_{i=1}^N\delta_{X_t^{i,N}},\mu_t)^2\d t+\frac{\sum_{i=1}^N\E|X_0^i-X_0^{i,N}|^2}{Nt_0^{4l-1}}\right\}, \ \ 0< f\in \B_b((\R^{m+d})^k).\end{align*}
 Again using \cite[Theorem 1.4.2(2)]{Wbook} and \eqref{S1}, we derive \eqref{VVN} when $b$ is bounded. Finally, by the same approximation technique in {\bf(Step (ii))} in the proof of Theorem \ref{POC}(2), we obtain the desired result for general $b$.
 \end{proof}
\section*{}
{\bf Data Availability Statement} Data sharing not applicable to this article as no datasets were generated or
analysed during the current study.
\section*{Declarations}
{\bf Conflict of Interests} The authors declare that they have no conflict of interest.

\end{document}